\documentclass[12pt]{article}
\usepackage[utf8]{inputenc}
\usepackage{amsmath, amsfonts, amsthm, amssymb}
\usepackage{graphicx}
\usepackage[margin=1in]{geometry}
\usepackage{color}
\usepackage{array}
\usepackage{float}
\usepackage{hyperref}

\usepackage[english]{babel}
\usepackage[autostyle]{csquotes}

\usepackage{pgfplots}
\pgfplotsset{compat=1.9}
\usepackage{mathrsfs}
\usetikzlibrary{arrows}

\newtheorem{theorem}{Theorem}[section]
\newtheorem{lemma}[theorem]{Lemma}
\newtheorem{proposition}[theorem]{Proposition}
\newtheorem{corollary}[theorem]{Corollary}
\newtheorem*{subtheorem}{Theorem}
\newtheorem*{maintheorem}{Main Theorem}

\theoremstyle{definition}
\newtheorem{definition}[theorem]{Definition}
\newtheorem{example}[theorem]{Example}
\newtheorem{remark}[theorem]{Remark}

\numberwithin{equation}{section}

\newcommand{\Eq}[1]{\begin{align}#1\end{align}}

\title{On the correspondence between perfect matchings and compatible pairs for affine cluster algebra}
\author{Ivan Ip\footnote{
	  Department of Mathematics, Hong Kong University of Science and Technology\newline
	  Email: ivan.ip@ust.hk\newline		
	  Email: ndphan@connect.ust.hk}
	  , Duy Phan\footnotemark[1]}
\date{\today}

\begin{document}

\maketitle

\begin{abstract}
We study cluster algebra of affine type $A_1^{(1)}$ by using two methods including counting the numbers of perfect matchings on snake graphs and compatible pairs on maximal Dyck paths. We find that the sum of coefficients of the terms in the Laurent polynomials of these cluster variables are odd-indexed Fibonacci numbers. In addition, we prove that the numbers of non-decreasing Dyck paths of even lengths are also odd-indexed Fibonacci numbers. As a consequence, we define explicit bijective correspondences among three combinatorial models, including perfect matchings on the snake graph, compatible pairs on the maximal Dyck path, and non-decreasing Dyck paths of even lengths.
\end{abstract}

\section{Introduction}

\emph{Cluster algebra} was first introduced by Fomin and Zelevinsky due to their inspiration for solving total positivity problems \cite{fomin2002cluster}. A type of commutative ring known as cluster algebra is made up of generators known as \emph{cluster variables}. Cluster variables are produced from the initial $n$ cluster variables by a process known as \emph{mutations}. All cluster variables can be written as a rational function of $n$ initial cluster variables. Furthermore, in \cite{fomin2002cluster} Fomin and Zelevinsky also proved that these rational functions are actually Laurent polynomials.\\

In this paper, we consider the \emph{affine type} cluster algebra $\mathcal{A}(2,2)$ of rank $2$ and type $A_1^{(1)}$ with an initial cluster $(x_1,x_2)$ and mutations expressed as a recurrence relation
\Eq{
x_{n-1}x_{n+1}=x_n^2+1,\quad n\in\mathbb{Z}
} (see \cite{zelevinsky2007cluster}). This cluster algebra is also associated with the surface of an annulus with one marked point on each boundary circle. The curves on this surface, also known as \emph{arcs}, correspond to cluster variables. \emph{Triangulations} of this surface correspond to clusters, and changes of triangulations correspond to mutations.\\

In \cite{musiker2010cluster}, cluster algebra associated with oriented unpunctured Riemann surface with boundaries is studied, by counting the \emph{perfect matchings} on the \emph{snake graphs} obtained from arcs on the surface. On the other hand, in \cite{lee2014greedy}, it is shown that cluster variables in rank $2$ cluster algebra are \emph{greedy elements} of the form $x[a_1,a_2]$. The greedy elements are computed by counting \emph{compatible pairs} on the \emph{maximal Dyck path} $\mathcal{D}^{a_1 \times a_2}$ of a rectangle of size $a_1 \times a_2$.\\

In this paper, we use these two methods to show that the numbers of perfect matchings on the snake graph $G_{T,\gamma_{n+3}}$ and the number of compatible pairs on the maximal Dyck path $\mathcal{D}^{\left(n+1\right)\times n}$ are both \emph{Fibonacci numbers} $F_{2n+3}$. This new calculation implies that in the coefficient free case, the sum of the integer coefficients in the Laurent polynomials of these cluster variables are the odd-indexed Fibonacci numbers. As a result, we come up with the idea of creating a new bijective correspondence between these two combinatorial patterns. This bijective map is defined by the following theorem (see Theorem \ref{thm: phi}).

\begin{subtheorem}
\label{thm: sub}

The set map
\[\phi: \{\text{perfect matchings on } G_{T,\gamma_{n+3}}\}
\mapsto \{\text{compatible pairs on } \mathcal{D}^{\left(n+1\right)\times n}\}\]
\[P \to (S_1,S_2) \]
defined by the conditions
\begin{itemize}
    \item $u_i\in S_1 \iff A_{2i}A_{2i+1},B_{2i}B_{2i+1}\in P$ for any $i=0,1,...,n$.
    \item $v_i\in S_2 \iff A_{2i-1}A_{2i},B_{2i-1}B_{2i}\in P$ for any $i=1,2,...,n$.
\end{itemize}
is a bijective correspondence.
\end{subtheorem}

Next, we expand our interest in the combinatorial patterns involving the Dyck path and the odd-indexed Fibonacci numbers. We find a combinatorial pattern called \emph{nondecreasing Dyck paths} of even length. The number of nondecreasing Dyck paths of length $2n+4$ is also $F_{2n+3}$. This model has significance in \emph{computer science} and has been studied by Deutsch and Prodinger \cite{deutsch2003bijection}. They showed that it has many connections with other combinatorial models such as \emph{directed column-convex polynominoes} and \emph{ordered trees} of height at most three. Therefore, we include nondecreasing Dyck paths of even length in our study.\\

Note that here the term \enquote{Dyck path} is defined slightly differently in \enquote{maximal Dyck path} and \enquote{nondecreasing Dyck path}, see Remark \ref{rem: 2 dyck} for an explanation.\\

\begin{figure}[h]
\centering
\begin{tikzpicture}[line cap=round,line join=round,>=triangle 45,x=0.8cm,y=0.8cm]
    
    \draw [line width=1pt, color=black, ->] (0.8,1.4) -- (2.8,4.4);
    \draw [line width=1pt, color=black, ->] (2,0) -- (6,0);
    \draw [line width=1pt, color=black, <-] (7,1.8) -- (5,4.8);

    \draw [line width=1pt, color=black, <-] (0.5,2) -- (2.5,5);
    \draw [line width=1pt, color=black, <-] (2,-0.5) -- (6,-0.5);
    \draw [line width=1pt, color=black, ->] (7.7,1.8) -- (5.7,4.8);
    
    \draw [line width=1pt, color=black, fill=white] (0,0) circle (2);
    \draw [line width=1pt, color=black, fill=white] (8,0) circle (2);
    \draw [line width=1pt, color=black, fill=white] (4,6) circle (2);
    
\begin{scriptsize}
    \draw[color=black] (0,0.8) node {perfect matchings};
    \draw[color=black] (0,-0.2) node {on the snake graph};
    \draw[color=black] (0,-1.2) node {$G_{T,\gamma_{n+3}}$};

    \draw[color=black] (4,6.8) node {compatible pairs on};
    \draw[color=black] (4,5.8) node {the maximal Dyck path};
    \draw[color=black] (4,4.8) node {$\mathcal{D}^{(n+1) \times n}$};

    \draw[color=black] (8,0.8) node {nondecreasing};
    \draw[color=black] (8,-0.2) node {Dyck paths of length};
    \draw[color=black] (8,-1.2) node {$2n+4$};

    \draw[color=black] (2.1,2.8) node {$\phi$};
    \draw[color=black] (6,2.8) node {$\theta$};
    \draw[color=black] (4,0.3) node {$\psi$};

    \draw[color=black] (4,2.3) node {Fibonacci numbers};
    \draw[color=black] (4,1.7) node {$F_{2n+3}$};

\end{scriptsize}
\end{tikzpicture}
\caption{Main correspondence.}
\label{fig: main}
\end{figure}
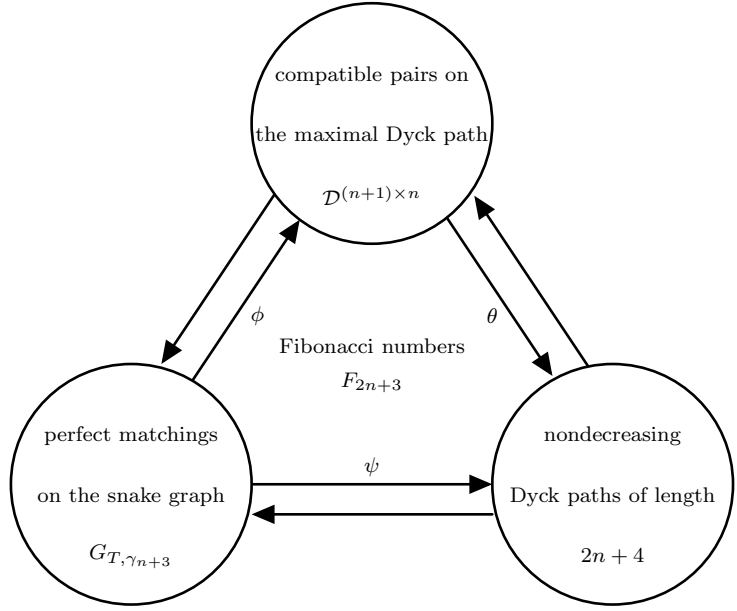

Finally, this paper constructs two more new maps $\theta$ and $\psi$ to complete the diagram (see Figure \ref{fig: main}) of correspondences among the three combinatorial models.

\begin{maintheorem}
\label{thm: main}
The maps $\phi$, $\theta$ and $\psi$ on Figure \ref{fig: main} give a bijective correspondence among perfect matchings on the snake graphs $G_{T,\gamma_{n+3}}$, compatible pairs on maximal Dyck paths $\mathcal{D}^{\left(n+1\right)\times n}$, and nondecreasing Dyck paths of length $2n+4$.
\end{maintheorem}
We expect our explicit construction in the case of $\mathcal{A}(2,2)$ cluster algebra considered here will shed light into finding the analogy of bijective correspondences in higher rank combinatorial models and discrete dynamical systems \cite{FZ,Kedem,Nakanishi}, in particular to cluster algebra arising from other surfaces and the more general canonical theta bases constructed in \cite{GHKK}, in which the greedy basis is a special case \cite{CGMMRSW}. A $q$-graded version of the correspondence may also be studied in relation to quantum cluster algebra and their quantum theta basis \cite{DM}.\\

The paper is organized as follows. We recall the general definition of cluster algebra in Section \ref{sec2.1} and then cluster algebra from surfaces in Section \ref{sec2.2}. In Section \ref{sec2.3}, we recall the construction of snake graphs and expansion formula for cluster variables in terms of perfect matchings on these graphs. In Section \ref{sec2.4}, we recall the definition of compatible pairs on maximal Dyck paths. We give recursion relations between an extended sequence of cluster variables and count the number of perfect matchings related to cluster algebra on an annulus in Section \ref{sec3}. In Section \ref{sec4} and Section \ref{sec5}, we prove that the number of compatible pairs on the maximal Dyck path and nondecreasing paths of length $2n$ are both odd-indexed Fibonacci numbers. In Section \ref{sec6}, we give the proof of the Main Theorems.

\section*{Acknowledgment}
This study is conducted under the Undergraduate Research Opportunities Program (UROP) at The Hong Kong University of Science and Technology. 
The first author is supported by the Hong Kong RGC General Research Funds [GRF \#16305122].

\section{Preliminaries}
\label{sec2}
In this section, we recall some basic definitions and properties about cluster algebra arising from surfaces and their expansion formulas using snake graphs, as well as the greedy basis of rank $2$ cluster algebra.

\subsection{Cluster algebra}
\label{sec2.1}

We review the construction of cluster algebra, first introduced by Fomin and Zelevinsky \cite{fomin2002cluster}. Our terminologies are based on \cite{fomin2007cluster}, that are rewritten in \cite{lee2013positivity,schiffler2010cluster,williams2014cluster}.\\

To define cluster algebra $\mathcal{A}$, we must first fix its ground ring.

\begin{definition}
A \emph{semifield} $(\mathbb{P}, \oplus, \circ)$ is  an abelian multiplicative group $(\mathbb{P}, \circ)$ together with a binary operation $\oplus$ such that
\[ \oplus: \mathbb{P} \times \mathbb{P} \longrightarrow \mathbb{P}\]
\begin{equation}
    (p,q) \mapsto p \oplus q
\end{equation}
is commutative, associative, and distributive:
\begin{equation}
    p \circ (q \oplus r) = p \circ q \oplus p \circ r
\end{equation}
\end{definition}

\begin{remark}
The binary operation $\oplus$ may not be invertible.
\end{remark}

Let $\mathbb{Z} \mathbb{P}$ be the group ring of $(\mathbb{P}, \circ)$, which is torsion free. Hence we can define the field of rational functions $\mathcal{F} = \mathbb{Q}\mathbb{P}(x_1,x_2,...,x_n)$ in $n$ variables with coefficients in $\mathbb{Q}\mathbb{P}$.

\begin{definition}
    A \emph{seed} is a triple $\Sigma =(\textbf{x},\textbf{y},B)$ where
    \begin{itemize}
        \item $\textbf{x}= \{x_1,x_2,...,x_n\}$ is a transcendental basis of $\mathcal{F}$ over $\mathbb{Q}\mathbb{P}$.
        \item $\textbf{y} =\{y_1,y_2,...,y_n\}$ is an $n$-tuple of elements $y_i \in \mathbb{P}$.
        \item $B=(b_{ij})$ is a skew-symmetrizable $n \times n$ integer matrix.
    \end{itemize}
\end{definition}

The set $\textbf{x}$ is called a \emph{cluster} and its elements $x_i$ are the \emph{cluster variables}. The set $\textbf{y}$ is called the \emph{coefficient tuple} and $B$ is called the \emph{exchange matrix}.\\

Throughout the paper, for any $x\in \mathbb{R}$ we will use the notation 
\Eq{
[x]_{+}:= \max(x,0).
}
\begin{definition}
A \emph{mutation} $\mu_k(\textbf{x}, \textbf{y},B)$ is a new seed $\Sigma' := (\textbf{x}', \textbf{y}',B')$ such that
\begin{itemize}
\item $\mathbf{x}':=(\mathbf{x} \setminus \{x_k\}) \cup \{x_k'\}$ where $x_k' \in \mathcal{F}$ is determined by the \emph{exchange relation}
\begin{equation}
    x_k':= \frac{y_k \Pi x_i^{[b_{ik}]_{+}} + \Pi x_i^{[-b_{ik}]_{+}}}{(y_{k} \oplus 1)x_k},
\end{equation}

\item $\mathbf{y}':=(y_1',y_2',...,y_n')$ where
\begin{equation}
    y_j':= \left\{
\begin{array}{ll}
  y_k^{-1}   & \text{ if } j=k,  \\
  y_jy_k^{[b_{kj}]_{+}}(y_k \oplus 1)^{-b_{kj}} & \text{ if } j \neq k,
\end{array}
\right.
\end{equation}

\item $B':=(b_{ij}')$ where
\begin{equation}
    b_{ij}':= \left\{
\begin{array}{ll}
  -b_{ij}   & \text{ if } i=k \text{ or } j=k,  \\
   b_{ij}+ [-b_{ik}]_{+} b_{kj} +b_{ik}[b_{kj}]_{+} & \text{ otherwise.}
\end{array}
\right.
\end{equation}
\end{itemize}
\end{definition}

\begin{definition}
Two seeds $\Sigma_1$ and $\Sigma_2$ are called \emph{mutation equivalent} if there is a finite sequence of mutations $\mu= \mu_{i_s}...\mu_{i_2}\mu_{i_1}$ such that $\mu\Sigma_1=\Sigma_2$.
\end{definition}

Let $\mathcal{X}$ be the union of all clusters $\mathbf{x}'$ belonging to a seed $(\mathbf{x}',\mathbf{y}',B')$ that is mutation equivalent to the initial seed $(\mathbf{x},\mathbf{y}, B)$. Now, we are ready to define cluster algebra.

\begin{definition}
The \emph{cluster algebra} $\mathcal{A}(\mathbf{x},\mathbf{y},B)$ is the $\mathbb{Z} \mathbb{P}$-subalgebra of the field $\mathcal{F}$ generated by the set of all cluster variables
\begin{equation}
     \mathcal{A}(\mathbf{x},\mathbf{y},B):= \mathbb{Z}\mathbb{P}[\mathcal{X}]
\end{equation}
We say that the cluster algebra is \emph{coefficient free} if $\mathbb{P}=1$, in which $\mathbb{ZP}=\mathbb{Z}$.
\end{definition}

\begin{definition}
The \emph{tropical semifield} $(\text{Trop}(u_1,...,u_m), \oplus, \circ)$ is defined to be the semifield of monomials in independent variables $u_j$’s with usual multiplication and semifield addition $\oplus$ given by
\begin{equation}
    \prod_{j} u_j^{a_j} \oplus \prod_{j} u_j^{b_j} = \prod_{j} u_j^{\min(a_j,b_j)}.
\end{equation}
 Note that the group ring of $\text{Trop}(u_1,...,u_m)$ is isomorphic to the ring of Laurent polynomials in the variables $u_j$.
\end{definition}

\begin{definition}
A cluster algebra $\mathcal{A}(\mathbf{x},\mathbf{y},B)$ is said to have \emph{principal coefficients} if the coefficient semifield $\mathbb{P}= \text{Trop}(y_1,y_2,...,y_n)$ is the tropical semifield with the initial coefficient tuple $\mathbf{y}=\{y_1,y_2,...,y_n\}$ as a set of generators. In this case, we say that the corresponding cluster algebra is of \emph{geometric type}.
\end{definition}

\subsection{Cluster algebra on surfaces}
\label{sec2.2}

We recall some facts about cluster algebras associated with unpunctured surfaces \cite{fomin2008positivity}. The reader can also refer to \cite{musiker2010cluster,musiker2011positivity,schiffler2010cluster}.\\

\begin{definition}
A \emph{bordered surface with marked points} is a pair $(S, M)$ such that $S$ is a connected oriented two-dimensional Riemann surface with boundary, and the set of \emph{marked points} $M$ is a nonempty finite set in the closure of $S$ with at least one marked point on each boundary component. Marked points in the interior of $S$ are called \emph{punctures}. The surface $(S, M)$ with all marked points lying on the boundary of $S$ are called \emph{unpunctured surface}.
\end{definition}

In the following, we will only consider unpunctured surfaces.

\begin{definition}
 An \emph{arc} $\gamma$ in $(S,M)$ is a curve in $S$, considered up to isotopy, such that
\begin{itemize}
    \item the endpoints of $\gamma$ belongs to $M$;
    \item except for the endpoints, $\gamma$ is disjoint from $M$ and from the boundary of $S$;
    \item $\gamma$ does not cut out an unpunctured 1-gon or an unpunctured 2-gon;
    \item $\gamma$ does not cross itself, except that its endpoints may coincide;
\end{itemize}
A \emph{boundary arc} is a curve that connects two marked points and lies entirely on the boundary of $S$ without passing through a third marked point. Two arcs are \emph{compatible} if they do not cross each other. A \emph{triangulation} $T$ of an unpunctured surface $(S, M)$ is a maximal collection of compatible arcs. 
The number of arcs $n$ in a triangulation of $S$ is called the \emph{rank} of the surface.
\end{definition}

\begin{example}
Figure \ref{fig: triangulation example} shows the surface $S$ of an annulus with a set of marked points $M$ including 1 marked point on the inner circle and 2 marked points on the outer circle. A triangulation is given by $T=\{\tau_1,\tau_2,\tau_3\}$ with the boundary arcs $\tau_4,\tau_5,\tau_6$.
\end{example}

\begin{figure}[H]
\centering
\begin{tikzpicture}[line cap=round,line join=round,>=triangle 45,x=0.7cm,y=0.7cm]
    \draw [line width=1pt] (0,0) circle (2);
    \draw [line width=1pt] (0,0) circle (5);
    \draw plot [smooth, tension=1] coordinates {(0,5) (-4,0) (0,-4) (4,0) (0,5)};
    \draw plot [smooth, tension=1] coordinates {(0,5) (-3,0) (0,-2)};
    \draw plot [smooth, tension=1] coordinates {(0,5) (3,0) (0,-2)};
    
\begin{scriptsize}
    \draw [fill=black] (0,5) circle (2pt);
    \draw [fill=black] (0,-2) circle (2pt);
    \draw [fill=black] (0,-5) circle (2pt);
    \draw[color=black] (0,-3.8) node {$\tau_{1}$};
    \draw[color=black] (-1,3) node {$\tau_{2}$};
    \draw[color=black] (1,3) node {$\tau_{3}$};
    \draw[color=black] (0,1.7) node {$\tau_{4}$};
    \draw[color=black] (-5.3,0) node {$\tau_{5}$};
    \draw[color=black] (5.3,0) node {$\tau_{6}$};

\end{scriptsize}
\end{tikzpicture}
\caption{A triangulation $T=\{\tau_1, \tau_2, \tau_3\}$ and boundary arcs $\tau_4, \tau_5, \tau_6$.}
\label{fig: triangulation example}
\end{figure}
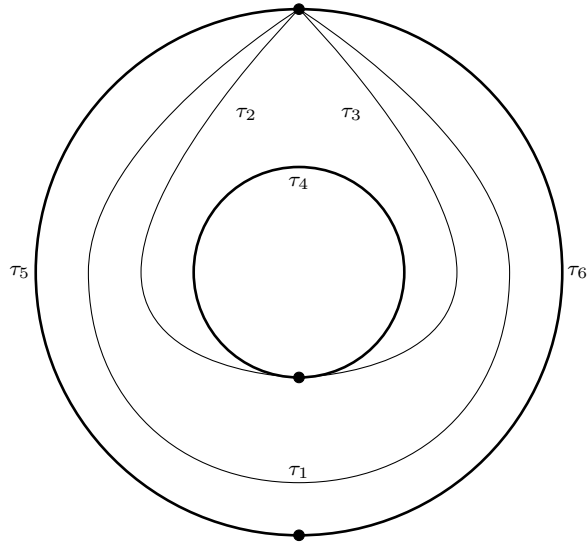

It is well known that any triangulations can be obtained from each other by a finite sequence of \emph{flips}. Each flip replaces a single arc $\gamma$ in $T$ by a unique new arc $\gamma' \neq \gamma$ such that
\begin{equation}
     T'=(T \setminus \{ \gamma\})  \cup \{\gamma'\} 
\end{equation}
is another triangulation (see Figure \ref{fig: flip}).

\begin{figure}[H]
\centering
\begin{tikzpicture}[line cap=round,line join=round,>=triangle 45,x=0.7cm,y=0.7cm]
    \draw plot [smooth] coordinates {(0,4) (0.2,3) (0,0)};
    \draw plot [smooth] coordinates {(0,0) (1,0.2) (4,0)};
    \draw plot [smooth] coordinates {(4,0) (3.8,1) (4,4)};
    \draw plot [smooth] coordinates {(4,4) (3,3.8) (0,4)};
    \draw (0,4) -- (4,0);

    \draw plot [smooth] coordinates {(8,4) (8.2,3) (8,0)};
    \draw plot [smooth] coordinates {(8,0) (9,0.2) (12,0)};
    \draw plot [smooth] coordinates {(12,0) (11.8,1) (12,4)};
    \draw plot [smooth] coordinates {(12,4) (9,3.8) (8,4)};
    \draw (8,0) -- (12,4);

    \draw [->] (5,2) -- (7,2);
    
\begin{scriptsize}
    \draw (6,2.5) node {flip $\gamma$};
    \draw (2.5,2) node {$\gamma$};
    \draw (10.5,2) node {$\gamma'$};
    \draw (-0.2,2) node {$\sigma_1$};
    \draw (4.2,2) node {$\sigma_2$};
    \draw (7.8,2) node {$\sigma_1$};
    \draw (12.2,2) node {$\sigma_2$};

    \draw (2,-0.2) node {$\rho_2$};
    \draw (2,4.2) node {$\rho_1$};
    \draw (10,-0.2) node {$\rho_2$};
    \draw (10,4.2) node {$\rho_1$};

    \draw [fill=black] (0,0) circle (2pt);
    \draw [fill=black] (0,4) circle (2pt);
    \draw [fill=black] (4,0) circle (2pt);
    \draw [fill=black] (4,4) circle (2pt);
    \draw [fill=black] (8,0) circle (2pt);
    \draw [fill=black] (8,4) circle (2pt);
    \draw [fill=black] (12,0) circle (2pt);
    \draw [fill=black] (12,4) circle (2pt);

\end{scriptsize}
\end{tikzpicture}
\caption{A flip.}
\label{fig: flip}
\end{figure}

\begin{remark}
On an unpunctured surface, there always exists a flip for any single arc $\gamma$. Moreover, there is no self-folded triangle in both triangulation $T$ and $T'$.
\end{remark}

We are now ready to define the cluster algebra associated to the surface $(S, M)$. Fix any triangulation $T=\{\tau_1, \tau_2, ..., \tau_n \}$ with $n$ arcs on $(S,M)$, and denote the $m$ boundary arcs of the surface by $\tau_{n+1}, \tau_{n+2},...,\tau_{n+m}$.

\begin{definition}
For any triangle $\bigtriangleup$ in $T$ define a \emph{minor matrix} $B^{\bigtriangleup} = (b_{ij}^{\bigtriangleup})_{i,j=1}^n$ by
\begin{equation}
b_{ij}^{\bigtriangleup}:= \left\{
\begin{array}{lll}
  1   & \text{if $\tau_i$ and $\tau_j$ are sides of $\bigtriangleup$ with $\tau_j$ following $\tau_i$ in the counter-clockwise order}  \\
  -1  & \text{if $\tau_i$ and $\tau_j$ are sides of $\bigtriangleup$ with $\tau_j$ following $\tau_i$ in the clockwise order}  \\
  0   & \text{otherwise}
\end{array}
\right.
\end{equation}
The \emph{exchange matrix} $B_T = (b_{ij})_{i,j=1}^n$ is defined by $b_{ij} = \sum_{\bigtriangleup\in T} b_{ij}^{\bigtriangleup}$, where the sum is taken over all triangles in $T$.
\end{definition}

\begin{definition}
A \emph{cluster algebra on surface} $\mathcal{A}(\mathbf{x}_T, \mathbf{y}_T, B_T)$ is defined by the cluster algebra with principal coefficients for a fixed triangulation $T$.

In other words, $\mathcal{A}(\mathbf{x}_T, \mathbf{y}_T, B_T)$ is given by the initial seed $\Sigma_T=(\mathbf{x}_T, \mathbf{y}_T, B_T)$ where
\begin{itemize}
    \item the initial cluster $\mathbf{x}_T= \{x_{\tau_1}, x_{\tau_2}, ...,x_{\tau_n}\}$ is the cluster associated to the triangulation,
    \item the initial coefficient vector $\mathbf{y}_T = \{y_1, y_2, ..., y_n\}$ is the vector of generators of $\mathbb{P} = \text{Trop}(y_1,y_2,...,y_n)$, and
    \item the exchange graph $B_T$ defined above.
\end{itemize}
\end{definition}

For the boundary arcs we will set $x_{\tau_k} = 1$, for $k=n+1, n+2, ..., n+m$. For each $k = 1, 2, ..., n$, there is a unique quadrilateral in $T \setminus \{\tau_k\}$ in which $\tau_k$ is one of the diagonals. Let $\tau_k'$ denote the other diagonal in that quadrilateral. Define the flip $\mu_k(T)$ to be the triangulation $(T \setminus \{\tau_k\}) \cup \{\tau_k'\}$. We refer to \cite{musiker2010cluster} for the following property.

\begin{proposition}
The mutation $\mu_k$ of the seed $\Sigma_T$ in the cluster algebra $\mathcal{A}$ corresponds to the flip $\mu_k$ of the triangulation $T$ in the following sense:
\begin{itemize}
    \item the matrix $\mu_k(B_T)$ is the matrix corresponding to the triangulation $\mu_k(T)$,
    \item the cluster $\mu_k(\mathbf{x}_T)$ is $(\mathbf{x}_T \setminus \{x_{\tau_k}\}) \cup \{x_{\tau_k}'\}$,
    \item the corresponding exchange relation is given by
    \begin{equation}
         x_{\tau_k}x_{\tau_k}' = x_{\rho_1} x_{\rho_2}y^{+} + x_{\sigma_1}x_{\sigma_2} y^{-}
    \end{equation}
    where $y^{+}, y^{-}$ are some coefficients in $\mathbb{P}$, and $\rho_1, \sigma_1, \rho_2, \sigma_2$ are the sides of the quadrilateral in which $\tau_k$ and $\tau_k'$ are the diagonals, such that $(\rho_1, \rho_2)$ and $(\sigma_1, \sigma_2)$ lie on opposite sides respectively (see Figure \ref{fig: flip}).
\end{itemize}
\end{proposition}

\subsection{Expansion formula}
\label{sec2.3}

We recall the expansion formula for the cluster variables in terms
of perfect matchings of a graph that is constructed recursively using \enquote{tiles}. For further details, the reader is referred to \cite{musiker2010cluster}.\\

\begin{definition}
A \emph{tile} $\overline{S}_k$ is a planar four-vertex graph with five weighted edges having the shape of two equilateral triangles that share one edge, oriented as in Figure \ref{fig: tile}. The parallelogram $S_k$ is constructed by removing the diagonal from $\overline{S}_k$.
\end{definition}

Now let $T$ be a triangulation of the unpunctured surface $(S, M)$. If $\tau_k \in T$ is an interior arc, then $\tau_k$ lies in precisely two triangles in $T$, hence $\tau_k$ is the diagonal of a unique quadrilateral $Q_{\tau_k}$ in $T$ with 4 edges $\tau_a,\tau_b,\tau_c,\tau_d$. We associate to this quadrilateral a tile $\overline{S}_k$ by assigning the weight $x_k$ (as a variable) to the diagonal and the weights $x_a,x_b,x_c,x_d$ to the sides of $\overline{S}_k$ (see Figure \ref{fig: tile}).\\

\begin{figure}[H] 
\centering
\begin{tikzpicture}[line cap=round,line join=round,>=triangle 45,x=1cm,y=1cm]
    \draw (0,0) -- (2,0);
    \draw (0,0) -- (1,2);
    \draw (2,0) -- (1,2);
    \draw (1,2) -- (3,2);
    \draw (2,0) -- (3,2);
    
\begin{scriptsize}
    \draw (1,-0.2) node {$x_a$};
    \draw (2.7,1) node {$x_b$};
    \draw (2,2.2) node {$x_c$};
    \draw (0.3,1) node {$x_d$};
    \draw (1.7,1) node {$x_k$};

    \draw [fill=black] (0,0) circle (2pt);
    \draw [fill=black] (2,0) circle (2pt);
    \draw [fill=black] (1,2) circle (2pt);
    \draw [fill=black] (3,2) circle (2pt);

\end{scriptsize}
\end{tikzpicture}
\caption{The tile $\overline{S}_k$}
\label{fig: tile}
\end{figure}
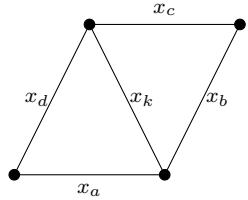

Let $\gamma\notin T$ be an arbitrary arc joining two endpoints $s,t$. Choose an orientation of $\gamma$ and assume $\gamma$ cuts the arcs in $T$ at the intersections (allowing repeats)
\begin{equation}
s:=p_0, p_1,p_2,...,p_d, p_{d+1}:=t,
\end{equation}
let $i_1,i_2,...,i_d$ be the indices such that $p_k$ lies on the arc $\tau_{i_k} \in T$. For $0\leq k\leq d$, let $\gamma^k$ denote the segment of the path $\gamma$ from the point $p_k$ to the point $p_{k+1}$. Each $\gamma^k$ lies in exactly one triangle $\Delta_k$ in $T$. For $1\leq k\leq d-1$, $\Delta_k$ is formed by the arcs $\tau_{i_k},\tau_{i_{k+1}}$ and a third arc that we denote by $\tau_{[\gamma^k]}$ for some index $[\gamma^k]\in\{1,2,...,n+m\}$.

For $1\leq k\leq d-1$, we glue tiles $\overline{S}_{i_k}$ and $\overline{S}_{i_{k+1}}$ along the common edge $\tau_{[\gamma^k]}$ such that the triangle $\Delta_k=\{\tau_{i_k},\tau_{i_{k+1}}, \tau_{[\gamma^k]}\}$ is oriented differently in the two tiles (see Figure \ref{fig: glue}).\\

\begin{figure}[H]
\centering
\begin{tikzpicture}[line cap=round,line join=round,>=triangle 45,x=1cm,y=1cm]
    \draw[blue] (0,0) -- (2,0);
    \draw[blue] (1,1.98) -- (3,1.98);
    \draw[red] (1,2) -- (3,2);
    \draw[red] (2,4) -- (4,4);
    \draw[blue] (0,0) -- (1,2);
    \draw[red] (1,2) -- (2,4);
    \draw[blue] (2,0) -- (3,2);
        \draw[red] (3,2) -- (4,4);
    \draw[blue] (1,2) -- (2,0);
    \draw[red] (2,4) -- (3,2);

\begin{scriptsize}
    \draw (2,2.2) node {$x_{[\gamma^k]}$};
    \draw (1.3,1) node {$x_{i_{k}}$};
    \draw (1.3,3) node {$x_{i_{k}}$};
    \draw (2.9,1) node {$x_{i_{k+1}}$};
    \draw (2.9,3) node {$x_{i_{k+1}}$};

    \draw [fill=black] (0,0) circle (2pt);
    \draw [fill=black] (2,0) circle (2pt);
    \draw [fill=black] (1,2) circle (2pt);
    \draw [fill=black] (3,2) circle (2pt);
    \draw [fill=black] (2,4) circle (2pt);
    \draw [fill=black] (4,4) circle (2pt);

\end{scriptsize}
\end{tikzpicture}
\caption{Gluing tiles $\overline{S}_{i_k}$ and $\overline{S}_{i_{k+1}}$ along the edge weighted $x_{[\gamma^k]}$}
\label{fig: glue}
\end{figure}
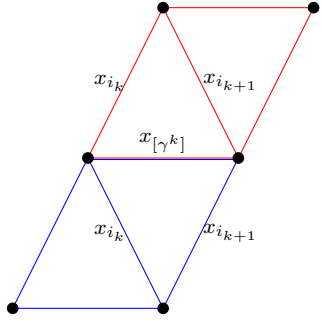

\begin{definition}
The \emph{full graph} $\overline{G}_{T, \gamma}$ is constructed by gluing tiles $\overline{S}_{{i_1}}, \overline{S}_{{i_2}},...,\overline{S}_{{i_d}}$ together with the procedure above, 
such that the first tile have the orientation induced from the surface, and the starting point $p_0$ of $\gamma$ lies in the southwest corner of the first tile. The \emph{snake graph} $G_{T, \gamma}$ is obtained from the full graph by removing the diagonal $\tau_{i_k}$ of each tile $\overline{S}_{{i_k}}$.
\end{definition}

\begin{definition}
A \emph{perfect matching} $P$ of a graph is a subset of the edges so that each vertex belongs to exactly one edge in $P$. The \emph{weight} $w(P)$ of a perfect matching $P$ of $G_{T, \gamma}$ is the product of all the weights assigned to the edges in $P$.
\end{definition}

There are precisely two perfect matchings $P_{+}$ and $P_{-}$ of $G_{T, \gamma}$ that contain only the boundary edges of $G_{T, \gamma}$. The matching $P_{-}$ contains the western (left-most) edge of the bottom-most tile on $G_{T, \gamma}$, while $P_{+}$ contains the southern (bottom) edge. 

\begin{definition}
The \emph{symmetric difference} of a perfect matching $P$ is a subgraph $G_P$ of $G_{T, \gamma}$ given by
\begin{equation}
    G_P:= P_- \ominus P = (P_- \cup P)\setminus (P_- \cap P)
\end{equation}
\end{definition}

Given an unpunctured surface $(S,M)$ and a triangulation $T=\{\tau_1, \tau_2,...,\tau_n\}$, we consider the cluster algebra $\mathcal{A}:=\mathcal{A}(\mathbf{x}_T, \mathbf{y}_T, B_T)$ as defined in Section \ref{sec2.2}. Each arc $\gamma$ in $(S,M)$ corresponds to a cluster variable $x_{\gamma}$ in $\mathcal{A}$, with the arcs in $T$ corresponding to the initial cluster $\mathbf{x}_T$. In the following we will denote the initial cluster variables by $x_k:=x_{\tau_k}$.\\

We can now state the following theorems \cite{musiker2010cluster} about cluster expansion.

\begin{theorem} \label{thm: y formula}
The set $P_- \ominus P$ is the set of boundary edges of a (possibly disconnected) subgraph $G_P$ of $G_{T,\gamma}$ which is a union of tiles
\begin{equation}
    G_P= \bigcup_{j \in J} S_j.
\end{equation}
\end{theorem}

Define a monomial in $y_T$ by
\begin{equation} \label{eq: y(P)}
    y(P):= \prod_{j \in J} y_{i_j}.
\end{equation}

\begin{theorem} \label{thm: x formula}
Any cluster variable can be expressed as a Laurent polynomial of the initial variables by
\begin{equation}
    x_{\gamma}= \sum_{P} \frac{w(P) y(P)}{x_{i_1}x_{i_2}...x_{i_d}}
\end{equation}
where the sum is over all perfect matchings $P$ of $G_{T, \gamma}$, $w(P)$ is the weight of $P$, and $y(P)$ is defined as in (\ref{eq: y(P)}).\\
\end{theorem}

\subsection{Maximal Dyck path and compatible pairs}
\label{sec2.4}

Any rank $2$ cluster algebra of finite or affine type has a $\mathbb{Z}$-basis that includes elements called \emph{indecomposable positives} \cite{sherman2003positivity}. A special family called \emph{greedy elements} is first introduced in \cite{lee2014greedy}. These elements are proven to be indecomposable positives, hence form a $\mathbb{Z}$-basis of the cluster algebra. There are several expressions of greedy elements, but this subsection only introduces the one mentioned in \cite{lee2014greedy}.\\

In the following, a rank $2$ cluster algebra is of the form $\mathcal{A}:=\mathcal{A}(b,c)$ defined by the exchange matrix $B=\begin{pmatrix}0&b\\-c&0\end{pmatrix}$, where $b,c\in \mathbb{N}$. It is generated by the initial cluster $\{x_1,x_2\}$ and the exchange relations for $n \in \mathbb{Z}$
\begin{equation}
    x_{n-1}x_{n+1}=
\left\{
\begin{array}{ll}
    x_n^b +1   & \text{for $n$ odd},\\
    x_n^c  +1  & \text{for $n$ even}.\\
\end{array}
\right.
\end{equation}
To express greedy elements in $\mathcal{A}(b,c)$, we need some terminologies.

\begin{definition} 
For $a_1,a_2 \in \mathbb{Z}_{\geq 0}$, a \emph{Dyck path} of dimension $a_1 \times a_2$ is a lattice path going from $(0,0)$ to $(a_1,a_2)$ that just go up or to the right, but is never higher than the straight line joining $(0, 0)$ and $(a_1,a_2)$. A Dyck path is \emph{maximal} if it is not lower than any other Dyck path.
\end{definition}

In the dimension $a_1 \times a_2$, the maximal Dyck path is unique, denoted by $\mathcal{D} = \mathcal{D}^{a_1 \times a_2}$. Let $\mathcal{D}_1 = \{ u_1, . . ., u_{a_1} \}$ be the set of horizontal edges of $\mathcal{D}$ indexed from left to right, and $\mathcal{D}_2 = \{ v_1, . . ., v_{a_2} \}$ the set of vertical edges of $\mathcal{D}$ indexed from bottom to top.\\

In the following, it is convenient to regard $(0, 0)$ and $(a_1, a_2)$ as the same point.

\begin{definition}
For any points $A$ and $B$ on $\mathcal{D}$, the \emph{subpath $AB$} is the path of $\mathcal{D}$ starting from $A$ and going in the Northeast direction until it reaches $B$ (if we reach $(a_1, a_2)$ first, we loop back to $(0, 0)$).
\end{definition}

\begin{remark}
In the special case $A = B$, define the subpath $AA$ to be the path that starts from A until it reaches $(a_1, a_2)$, loops back to $(0,0)$, and then ends at $A$.
\end{remark}

We denote by $(AB)_1$ the set of horizontal edges in $AB$ and by $(AB)_2$ the set of vertical edges in $AB$. Also, let $AB^{\circ}$ denote the set of lattice points on the subpath $AB$ excluding the endpoints $A$ and $B$.

\begin{example}
Figure \ref{fig: dyck} shows the maximal Dyck path $\mathcal{D}^{6 \times 4}$, where $\mathcal{D}_1=\{u_1,u_2,u_3,u_4,u_5,u_6\}$ and $\mathcal{D}_2=\{v_1,v_2,v_3,v_4\}$. We choose $A=(5,2), B=(2,1)$, with subpath $AB=(v_3,u_6,v_4,u_1,u_2,v_1)$ and the set of lattice points $AB^{\circ}=\{(5,3), (6,3), (6,4)=(0,0), (1,0), (2,0)\}$.
\end{example}

\begin{figure}[H]
\centering
\begin{tikzpicture}[line cap=round,line join=round,>=triangle 45,x=1.5cm,y=1.5cm]
    \draw (0,0) -- (6,0);
    \draw (0,1) -- (6,1);
    \draw (0,2) -- (6,2);
    \draw (0,3) -- (6,3);
    \draw (0,4) -- (6,4);

    \draw (0,0) -- (0,4);
    \draw (1,0) -- (1,4);
    \draw (2,0) -- (2,4);
    \draw (3,0) -- (3,4);
    \draw (4,0) -- (4,4);
    \draw (5,0) -- (5,4);
    \draw (6,0) -- (6,4);

    \draw [color = blue] (0,0) -- (6,4);

    \draw [line width=4pt] (0,0) -- (2,0);
    \draw [line width=4pt] (2,0) -- (2,1);
    \draw [line width=4pt] (2,1) -- (3,1);
    \draw [line width=4pt] (3,1) -- (3,2);
    \draw [line width=4pt] (3,2) -- (5,2);
    \draw [line width=4pt] (5,2) -- (5,3);
    \draw [line width=4pt] (5,3) -- (6,3);
    \draw [line width=4pt] (6,3) -- (6,4);
        
\begin{scriptsize}
    \draw (0.5,0.2) node {$u_1$};
    \draw (1.5,0.2) node {$u_2$};
    \draw (2.5,1.2) node {$u_3$};
    \draw (3.5,2.2) node {$u_4$};
    \draw (4.5,2.2) node {$u_5$};
    \draw (5.5,3.2) node {$u_6$};

    \draw (2.2,0.5) node {$v_1$};
    \draw (3.2,1.5) node {$v_2$};
    \draw (5.2,2.5) node {$v_3$};
    \draw (6.2,3.5) node {$v_4$};

    \draw (5.2,1.8) node {\large $A$};
    \draw (1.8,1.2) node {\large $B$};
    
\end{scriptsize}
\end{tikzpicture}
\caption{The maximal Dyck path of $\mathcal{D}^{6 \times 4}$}
\label{fig: max dyck path example}
\end{figure}
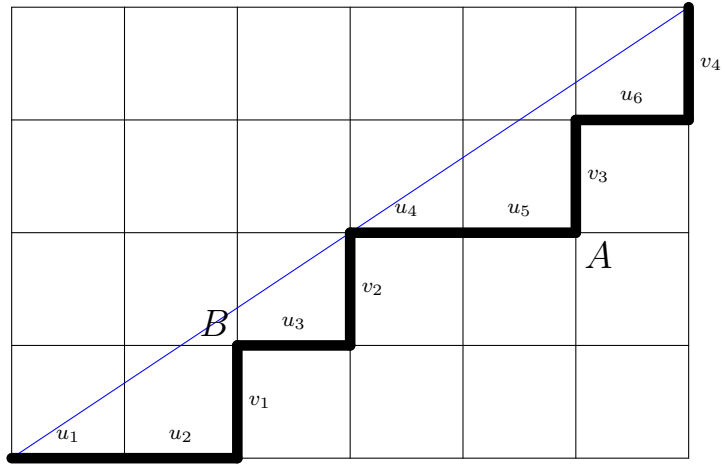

\begin{definition}\label{compatibledef}
A pair $(S_1,S_2)$ with $S_1 \subset \mathcal{D}_1$, $S_2 \subset \mathcal{D}_2$ is called \emph{compatible} if for every $u \in S_1, v \in S_2$, there exists $A \in EF^{\circ}$ ,where $E$ is the left endpoint of $u$ and $F$ is the upper endpoint of $v$, such that:
\begin{equation}
    |(AF)_1|=b|(AF)_2 \cap S_2| \text{ or } |(EA)_2|=c|(EA)_1 \cap S_1|
\end{equation}
\end{definition}

Now, we recall the expression of greedy elements mentioned in \cite{lee2014greedy}.

\begin{theorem}
For any $a_1,a_2 \in \mathbb{Z}$, the \emph{greedy element} $x[a_1,a_2] \in \mathcal{A}(b,c)$ is defined by
\begin{equation}
    x[a_1,a_2]:=x_1^{-a_1}x_2^{-a_2} \sum_{(S_1,S_2)} x_1^{b|S_2|}x_2^{c|S_1|}.
\end{equation}
where the sum is over all compatible pairs $(S_1,S_2)$ in $\mathcal{D}^{{[a_1]_+} \times {[a_2]_+}}$.
\end{theorem}

\begin{remark}
If $a_1 \leq 0$ and $a_2 \leq 0$, then by the definition, \Eq{x[a_1,a_2]=x_1^{-a_1}x_2^{-a_2}.}
\end{remark}

\begin{definition}
For any $p \in \mathbb{Z}$, we denote the \emph{standard permutation} for cluster variable
\begin{equation}
    \sigma_p(x_n)=x_{2p-n}.
\end{equation}
\end{definition}

We recall some propositions in \cite{lee2014greedy}.

\begin{proposition}
The set of all greedy elements \Eq{\mathcal{B}:=\{x[a_1,a_2]: a_1, a_2 \in \mathbb{Z}\}} forms a $\mathbb{Z}$-basis of $\mathcal{A}(b,c)$. It is called the \emph{greedy basis}.
\end{proposition}

\begin{proposition} \label{prop: sigma formula}
The greedy basis $\mathcal{B}$ is invariant under the action of any $\sigma_p$, $p\in\mathbb{Z}$. Specifically, $\sigma_p$ can be written as a composition of some automorphisms $\sigma_1$ and $\sigma_2$, where they act on the greedy elements as:
\Eq{
     \sigma_1(x[a_1,a_2])&:=x[a_1, c[a_1]_{+} -a_2],\\
    \sigma_2(x[a_1,a_2])&:=x[b[a_2]_{+} -a_1,a_2].
}
\end{proposition}

\section{Perfect matchings of the snake graph}
\label{sec3}

In this section, we study the cluster algebra associated to an annulus with one marked point on each boundary circle by using perfect matchings on the snake graphs mentioned in Section \ref{sec2.3}. We find some recurrence relations between cluster variables associated with the arcs.\\

Let $(S, M)$ be the annulus with one marked point on each of the two boundary components, and let $T= \{\tau_1, \tau_2 \}$ be the triangulation with boundary arcs $\tau_3, \tau_4$ shown in Figure \ref{fig: annulus}. Denote the marked points to be $M_1$ and $M_2$. Let $\mathcal{A}(\textbf{x}_T, \textbf{y}_T, B_T)$ be the cluster algebra associated to this triangulation with the initial cluster $\textbf{x}_T= \{x_1,x_2\}$ and the initial coefficient vector $\textbf{y}_T= \{y_1,y_2\}$. The exchange matrix of this triangulation is
\begin{equation}
    B_T = 
\begin{pmatrix}
0 & 2\\
-2 & 0
\end{pmatrix}.
\end{equation}

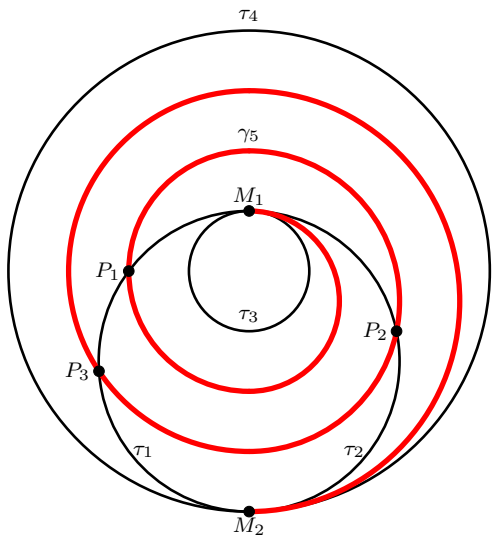
\begin{figure}
\centering
\begin{tikzpicture}[line cap=round,line join=round,>=triangle 45,x=0.4cm,y=0.4cm]

    \draw [line width=1pt] (9,0) circle (8);

    \draw [line width=1pt] (9,0) circle (2);

    \draw [line width=1pt] (9,-3) circle (5);
    
    \draw [line width=2pt][color=red] (9,2) arc (90:-90:3);
    \draw [line width=2pt][color=red] (9,-4) arc (270:90:4);
    \draw [line width=2pt][color=red] (9,4) arc (90:-90:5);
    \draw [line width=2pt][color=red] (9,-6) arc (270:90:6);
    \draw [line width=2pt][color=red] (9,6) arc (90:-90:7);
\begin{scriptsize}
    \draw [fill=black] (9,2) circle (2pt);
    \draw[color=black] (9,2.5) node {$M_{1}$};
    \draw [fill=black] (9,-8) circle (2pt);
    \draw[color=black] (9,-8.5) node {$M_{2}$};

    \draw [fill=black] (5,0) circle (2pt);
    \draw[color=black] (4.3,0) node {$P_{1}$};
    \draw [fill=black] (13.9,-2) circle (2pt);
    \draw[color=black] (13.2,-2) node {$P_{2}$};
    \draw [fill=black] (4.01,-3.33) circle (2pt);
    \draw[color=black] (3.3,-3.33) node {$P_{3}$};

    \draw[color=black] (5.5,-6) node {$\tau_{1}$};
    \draw[color=black] (12.5,-6) node {$\tau_{2}$};
    \draw[color=black] (9,-1.5) node {$\tau_{3}$};
    \draw[color=black] (9,8.5) node {$\tau_{4}$};
    \draw[color=black] (9,4.5) node {$\gamma_{5}$};
\end{scriptsize}
\end{tikzpicture}
\caption{Annulus with triangulation $T=\{\tau_1, \tau_2\}$ and the arc $\gamma_4$ (shown in red)}
\label{fig: annulus}
\end{figure}

\begin{definition}
For $k \geq 3$, define $\gamma_k$ to be the arc that
\begin{itemize}
    \item wraps clockwise around the inner circle $(k-2)$ times and cut $\tau_1, \tau_2$ at $k-2,k-3$ points respectively.
    \item starts from $M_1$ through the corner $(\tau_2,\tau_3)$ of the triangle $\Delta=\{\tau_1,\tau_2,\tau_3\}$.
    \item goes into the corner $(\tau_2,\tau_4)$ of the triangle $\Delta=\{\tau_1,\tau_2,\tau_4\}$ and then ends at $M_2$.
\end{itemize}
\end{definition}

The definition gives the following flips of triangulations:
\begin{itemize}
    \item flips $T=(\tau_1,\tau_2)$ along $\tau_1$, we get another triangulation $T'=(\tau_2, \gamma_3)$,
    \item flips $T'=(\tau_2,\gamma_3)$ along $\tau_2$, we get $T''=(\gamma_3, \gamma_4)$,
    \item for $k \geq 3$, flips $(\gamma_k,\gamma_{k+1})$ along $\gamma_k$, we get $(\gamma_{k+1}, \gamma_{k+2})$.
\end{itemize}

Next, for any $k \geq 3$, the full graph $\overline{G}_{T, \gamma_k}$ and snake graph $G_{T, \gamma_k}$ are constructed as in Section \ref{sec2.3}. They consist of $(2k-5)$ parallelograms stacked consecutively on top of each other. 

\begin{example}
Figure \ref{fig: annulus} shows the arc $\gamma_4$ cutting $\tau_1, \tau_2$ at $P_1,P_2,P_3$ in clockwise direction. Figure \ref{fig: full and snake} shows the graphs $\overline{G}_{T, \gamma_4}$ and $G_{T, \gamma_4}$.
\end{example}

We recall the expansion formula in Section \ref{sec2.3}:
\begin{equation}
    x_{\gamma_k}= \sum_{P} \frac{w(P) y(P)}{x_1^{k-2}x_2^{k-3}}
\end{equation}
where the sum is over all perfect matchings $P$ of $G_{T, \gamma_k}$.\\

The snake graphs $G_{T, \gamma_k}$ contain only an odd number of parallelograms. Therefore, for convenience in calculation, we introduce the following graphs.

\begin{definition}
The graphs $\overline{H}_1,\overline{H}_2,...$ is defined by
\begin{itemize}
    \item For $n \geq 0$, $\overline{H}_{2n+1} = \overline{G}_{T, \gamma_{n+3}}$
    \item For $n \geq 1$, $\overline{H}_{2n}$ is obtained from $\overline{G}_{T, \gamma_{n+3}}$ by deleting the top parallelogram.
\end{itemize}
\end{definition}

Similarly, the graph $H_n$ is obtained from $\overline{H}_n$ by deleting the diagonals of its parallelograms. Now the graphs $H_n$ have $n$ parallelograms. For even more convenience, we can draw $H_n$ horizontally as $n$ consecutive squares (see Figure \ref{fig: perfect}). The configuration is now good enough for calculation.\\

Inspired by the expansion formula, for $n \geq 1$, we introduce the following variables
\begin{equation}
    z_{2n-1}:= \sum_{P} \frac{w(P) y(P)}{x_1^n x_2^{n-1}}
\end{equation}
\begin{equation}
    z_{2n}:= \sum_{P} \frac{w(P) y(P)}{x_1^n x_2^n}
\end{equation}
where the sums are over all perfect matchings $P$ of $H_{2n-1}$ and $H_{2n}$ respectively.

\begin{remark}
The relationship between the cluster variables and the sequence $\{z_n\}$ is given by
\begin{equation}
    x_{\gamma_k}=z_{2k-5}.
\end{equation}
\end{remark}

\begin{example} \label{ex: z123}
Figure \ref{fig: full and snake} shows the graphs of $\overline{H}_2,H_2,\overline{H}_3=\overline{G}_{T, \gamma_4}, H_3=G_{T, \gamma_4}$. Then, applying the formula of $z_n$ on all perfect matchings of $H_n$, we have
\begin{equation}
    \begin{aligned}
z_1 (x_1 )
& = (x_3x_4)(y_1)\\
& + (x_2x_2)\\
\end{aligned}
\end{equation}
\begin{equation}
    \begin{aligned}
z_2 (x_1x_2) 
& = (x_3x_4x_3)(y_1)\\
& + (x_3x_1x_1)(y_1y_2)\\
& + (x_2x_2x_3)\\
\end{aligned}
\end{equation}
\begin{equation}
    \begin{aligned}
z_3 (x_1x_2x_1) 
& = (x_3x_4x_3x_4)(y_1y_1)\\
& + (x_3x_4x_2x_2)(y_1)\\
& + (x_3x_1x_1x_4)(y_1y_2y_1)\\
& + (x_2x_2x_3x_4)(y_1)\\
& + (x_2x_2x_2x_2).\\
\end{aligned}
\end{equation}
\end{example}

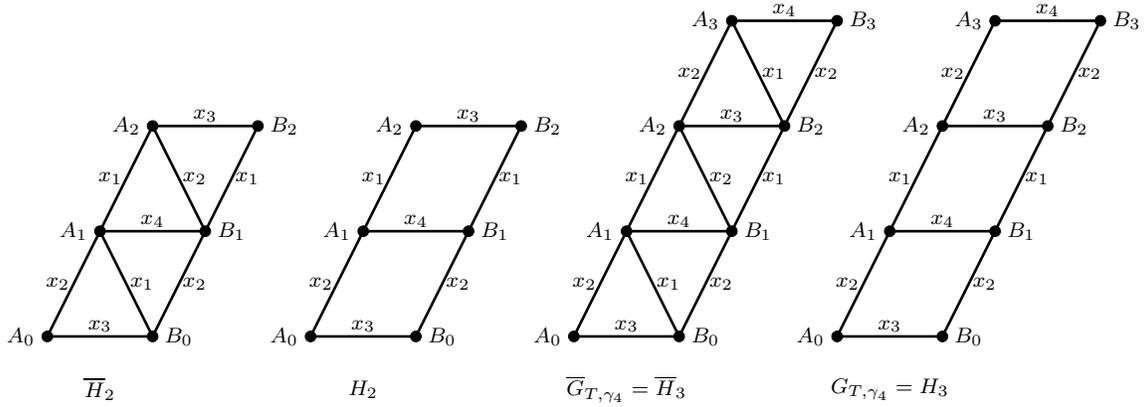
\begin{figure}[H]
\centering
\begin{tikzpicture}[line cap=round,line join=round,>=triangle 45,x=0.7cm,y=0.7cm]

    \draw [line width=1pt] (0,0) -- (2,0);
    \draw [line width=1pt] (1,2) -- (3,2);
    \draw [line width=1pt] (2,4) -- (4,4);
    \draw [line width=1pt] (0,0) -- (2,4);
    \draw [line width=1pt] (2,0) -- (4,4);
    \draw [line width=1pt] (1,2) -- (2,0);
    \draw [line width=1pt] (2,4) -- (3,2);
\begin{scriptsize}
    \draw[color=black] (1,-1) node {$\overline{H}_{2}$};
    
    \draw [fill=black] (0,0) circle (2pt);
    \draw[color=black] (-0.5,0) node {$A_{0}$};
    \draw [fill=black] (1,2) circle (2pt);
    \draw[color=black] (0.5,2) node {$A_{1}$};
    \draw [fill=black] (2,4) circle (2pt);
    \draw[color=black] (1.5,4) node {$A_{2}$};
    \draw [fill=black] (2,0) circle (2pt);
    \draw[color=black] (2.5,0) node {$B_{0}$};
    \draw [fill=black] (3,2) circle (2pt);
    \draw[color=black] (3.5,2) node {$B_{1}$};
    \draw [fill=black] (4,4) circle (2pt);
    \draw[color=black] (4.5,4) node {$B_{2}$};

    \draw[color=black] (1,0.2) node {$x_{3}$};
    \draw[color=black] (2,2.2) node {$x_{4}$};
    \draw[color=black] (3,4.2) node {$x_{3}$};
    \draw[color=black] (0.2,1) node {$x_{2}$};
    \draw[color=black] (2.8,1) node {$x_{2}$};
    \draw[color=black] (1.2,3) node {$x_{1}$};
    \draw[color=black] (3.8,3) node {$x_{1}$};
    \draw[color=black] (1.8,1) node {$x_{1}$};
    \draw[color=black] (2.8,3) node {$x_{2}$};
\end{scriptsize}

\begin{scope}[shift={(1,0)}]
    \draw [line width=1pt] (4,0) -- (6,0);
    \draw [line width=1pt] (5,2) -- (7,2);
    \draw [line width=1pt] (6,4) -- (8,4);
    \draw [line width=1pt] (4,0) -- (6,4);
    \draw [line width=1pt] (6,0) -- (8,4);
\begin{scriptsize}
    \draw[color=black] (5,-1) node {$H_{2}$};
    
    \draw [fill=black] (4,0) circle (2pt);
    \draw[color=black] (3.5,0) node {$A_{0}$};
    \draw [fill=black] (5,2) circle (2pt);
    \draw[color=black] (4.5,2) node {$A_{1}$};
    \draw [fill=black] (6,4) circle (2pt);
    \draw[color=black] (5.5,4) node {$A_{2}$};
    \draw [fill=black] (6,0) circle (2pt);
    \draw[color=black] (6.5,0) node {$B_{0}$};
    \draw [fill=black] (7,2) circle (2pt);
    \draw[color=black] (7.5,2) node {$B_{1}$};
    \draw [fill=black] (8,4) circle (2pt);
    \draw[color=black] (8.5,4) node {$B_{2}$};

    \draw[color=black] (5,0.2) node {$x_{3}$};
    \draw[color=black] (6,2.2) node {$x_{4}$};
    \draw[color=black] (7,4.2) node {$x_{3}$};
    \draw[color=black] (4.2,1) node {$x_{2}$};
    \draw[color=black] (6.8,1) node {$x_{2}$};
    \draw[color=black] (5.2,3) node {$x_{1}$};
    \draw[color=black] (7.8,3) node {$x_{1}$};
\end{scriptsize}
\end{scope}

\begin{scope}[shift={(2,0)}]
    \draw [line width=1pt] (8,0) -- (10,0);
    \draw [line width=1pt] (9,2) -- (11,2);
    \draw [line width=1pt] (10,4) -- (12,4);
    \draw [line width=1pt] (11,6) -- (13,6);
    \draw [line width=1pt] (8,0) -- (11,6);
    \draw [line width=1pt] (10,0) -- (13,6);
    \draw [line width=1pt] (10,0) -- (9,2);
    \draw [line width=1pt] (11,2) -- (10,4);
    \draw [line width=1pt] (12,4) -- (11,6);
\begin{scriptsize}
    \draw[color=black] (9,-1) node {$\overline{G}_{T, \gamma_{4}}=\overline{H}_3$};

    \draw [fill=black] (8,0) circle (2pt);
    \draw[color=black] (7.5,0) node {$A_{0}$};
    \draw [fill=black] (9,2) circle (2pt);
    \draw[color=black] (8.5,2) node {$A_{1}$};
    \draw [fill=black] (10,4) circle (2pt);
    \draw[color=black] (9.5,4) node {$A_{2}$};
    \draw [fill=black] (11,6) circle (2pt);
    \draw[color=black] (10.5,6) node {$A_{3}$};
    \draw [fill=black] (10,0) circle (2pt);
    \draw[color=black] (10.5,0) node {$B_{0}$};
    \draw [fill=black] (11,2) circle (2pt);
    \draw[color=black] (11.5,2) node {$B_{1}$};
    \draw [fill=black] (12,4) circle (2pt);
    \draw[color=black] (12.5,4) node {$B_{2}$};
    \draw [fill=black] (13,6) circle (2pt);
    \draw[color=black] (13.5,6) node {$B_{3}$};

    \draw[color=black] (9,0.2) node {$x_{3}$};
    \draw[color=black] (10,2.2) node {$x_{4}$};
    \draw[color=black] (11,4.2) node {$x_{3}$};
    \draw[color=black] (12,6.2) node {$x_{4}$};
    \draw[color=black] (8.2,1) node {$x_{2}$};
    \draw[color=black] (10.8,1) node {$x_{2}$};
    \draw[color=black] (9.2,3) node {$x_{1}$};
    \draw[color=black] (11.8,3) node {$x_{1}$};
    \draw[color=black] (10.2,5) node {$x_{2}$};
    \draw[color=black] (12.8,5) node {$x_{2}$};
    \draw[color=black] (9.8,1) node {$x_{1}$};
    \draw[color=black] (10.8,3) node {$x_{2}$};
    \draw[color=black] (11.8,5) node {$x_{1}$};
\end{scriptsize}
\end{scope}

\begin{scope}[shift={(3,0)}]
    \draw [line width=1pt] (12,0) -- (14,0);
    \draw [line width=1pt] (13,2) -- (15,2);
    \draw [line width=1pt] (14,4) -- (16,4);
    \draw [line width=1pt] (15,6) -- (17,6);
    \draw [line width=1pt] (12,0) -- (15,6);
    \draw [line width=1pt] (14,0) -- (17,6);
\begin{scriptsize}
    \draw[color=black] (13,-1) node {$G_{T, \gamma_{4}}=H_3$};

    \draw [fill=black] (12,0) circle (2pt);
    \draw[color=black] (11.5,0) node {$A_{0}$};
    \draw [fill=black] (13,2) circle (2pt);
    \draw[color=black] (12.5,2) node {$A_{1}$};
    \draw [fill=black] (14,4) circle (2pt);
    \draw[color=black] (13.5,4) node {$A_{2}$};
    \draw [fill=black] (15,6) circle (2pt);
    \draw[color=black] (14.5,6) node {$A_{3}$};
    \draw [fill=black] (14,0) circle (2pt);
    \draw[color=black] (14.5,0) node {$B_{0}$};
    \draw [fill=black] (15,2) circle (2pt);
    \draw[color=black] (15.5,2) node {$B_{1}$};
    \draw [fill=black] (16,4) circle (2pt);
    \draw[color=black] (16.5,4) node {$B_{2}$};
    \draw [fill=black] (17,6) circle (2pt);
    \draw[color=black] (17.5,6) node {$B_{3}$};
    
    \draw[color=black] (13,0.2) node {$x_{3}$};
    \draw[color=black] (14,2.2) node {$x_{4}$};
    \draw[color=black] (15,4.2) node {$x_{3}$};
    \draw[color=black] (16,6.2) node {$x_{4}$};
    \draw[color=black] (12.2,1) node {$x_{2}$};
    \draw[color=black] (14.8,1) node {$x_{2}$};
    \draw[color=black] (13.2,3) node {$x_{1}$};
    \draw[color=black] (15.8,3) node {$x_{1}$};
    \draw[color=black] (14.2,5) node {$x_{2}$};
    \draw[color=black] (16.8,5) node {$x_{2}$};
\end{scriptsize}
\end{scope}

\end{tikzpicture}
\caption{$\overline{H}_{2},H_2$; the full and snake graphs of the arcs $\gamma_4$}
\label{fig: full and snake}
\end{figure}

 Let $A_0,A_1,...,A_{k}$ and $B_0,B_1,...,B_{k}$ be vertices and $x_i$ be weights on edges of the graph $H_k$. The graph $P_{-}$ contains edges of the form $A_{2i}A_{2i+1}, B_{2i}B_{2i+1}$.

\begin{example}
Figure \ref{fig: perfect} shows perfect matchings $P$ on $H_{13}=G_{T, \gamma_9}$ and $P_1$ on $H_{12}$. The third and fourth graphs are their symmetric differences $H_{P}$ and $H_{P_1}$ with $P_-$.
\end{example}

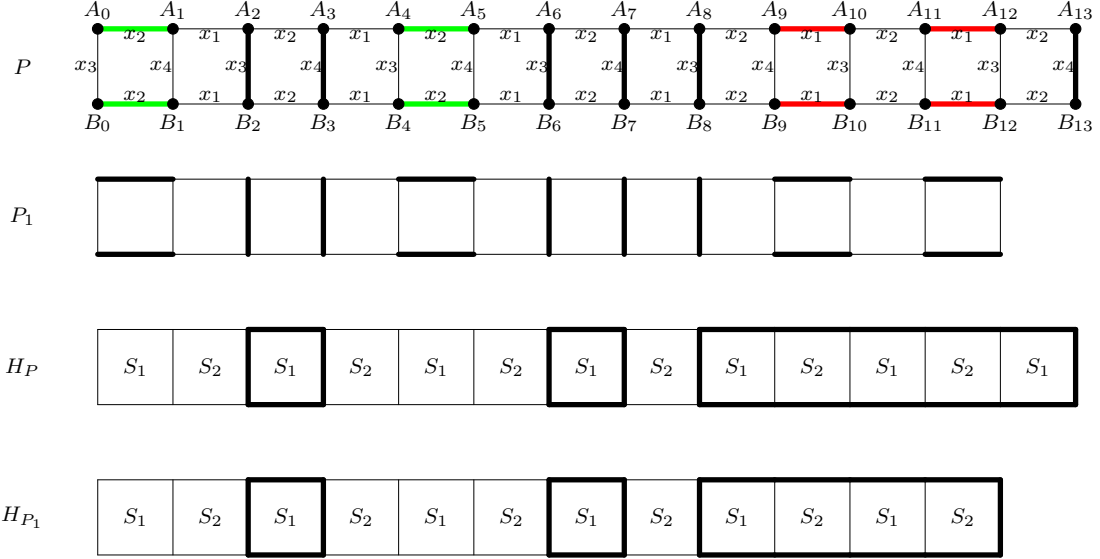
\begin{figure}[H]
\centering
\begin{tikzpicture}[line cap=round,line join=round,>=triangle 45,x=0.5cm,y=0.5cm]
    \draw [line width=0.2pt] (0,0) -- (26,0);
    \draw [line width=0.2pt] (0,2) -- (26,2);
    \draw [line width=0.2pt] (0,0) -- (0,2);
    \draw [line width=0.2pt] (2,0) -- (2,2);
    \draw [line width=0.2pt] (4,0) -- (4,2);
    \draw [line width=0.2pt] (6,0) -- (6,2);
    \draw [line width=0.2pt] (8,0) -- (8,2);
    \draw [line width=0.2pt] (10,0) -- (10,2);
    \draw [line width=0.2pt] (12,0) -- (12,2);
    \draw [line width=0.2pt] (14,0) -- (14,2);
    \draw [line width=0.2pt] (16,0) -- (16,2);
    \draw [line width=0.2pt] (18,0) -- (18,2);
    \draw [line width=0.2pt] (20,0) -- (20,2);
    \draw [line width=0.2pt] (22,0) -- (22,2);
    \draw [line width=0.2pt] (24,0) -- (24,2);
    \draw [line width=0.2pt] (26,0) -- (26,2);

    \draw [line width=2pt][color=green] (0,0) -- (2,0);
    \draw [line width=2pt][color=green] (0,2) -- (2,2);
    \draw [line width=2pt] (4,0) -- (4,2);
    \draw [line width=2pt] (6,0) -- (6,2);
    \draw [line width=2pt][color=green] (8,0) -- (10,0);
    \draw [line width=2pt][color=green] (8,2) -- (10,2);
    \draw [line width=2pt] (12,0) -- (12,2);
    \draw [line width=2pt] (14,0) -- (14,2);
    \draw [line width=2pt] (16,0) -- (16,2);
    \draw [line width=2pt][color=red] (18,0) -- (20,0);
    \draw [line width=2pt][color=red] (18,2) -- (20,2);
    \draw [line width=2pt] (26,0) -- (26,2);
    \draw [line width=2pt][color=red] (22,0) -- (24,0);
    \draw [line width=2pt][color=red] (22,2) -- (24,2);

\begin{scriptsize}
    \draw [color=black] (1,0.2) node {$x_2$};
    \draw [color=black] (3,0.2) node {$x_1$};
    \draw [color=black] (5,0.2) node {$x_2$};
    \draw [color=black] (7,0.2) node {$x_1$};
    \draw [color=black] (9,0.2) node {$x_2$};
    \draw [color=black] (11,0.2) node {$x_1$};
    \draw [color=black] (13,0.2) node {$x_2$};
    \draw [color=black] (15,0.2) node {$x_1$};
    \draw [color=black] (17,0.2) node {$x_2$};
    \draw [color=black] (19,0.2) node {$x_1$};
    \draw [color=black] (21,0.2) node {$x_2$};
    \draw [color=black] (23,0.2) node {$x_1$};
    \draw [color=black] (25,0.2) node {$x_2$};

    \draw [color=black] (1,1.8) node {$x_2$};
    \draw [color=black] (3,1.8) node {$x_1$};
    \draw [color=black] (5,1.8) node {$x_2$};
    \draw [color=black] (7,1.8) node {$x_1$};
    \draw [color=black] (9,1.8) node {$x_2$};
    \draw [color=black] (11,1.8) node {$x_1$};
    \draw [color=black] (13,1.8) node {$x_2$};
    \draw [color=black] (15,1.8) node {$x_1$};
    \draw [color=black] (17,1.8) node {$x_2$};
    \draw [color=black] (19,1.8) node {$x_1$};
    \draw [color=black] (21,1.8) node {$x_2$};
    \draw [color=black] (23,1.8) node {$x_1$};
    \draw [color=black] (25,1.8) node {$x_2$};

    \draw [color=black] (-0.3,1) node {$x_3$};
    \draw [color=black] (1.7,1) node {$x_4$};
    \draw [color=black] (3.7,1) node {$x_3$};
    \draw [color=black] (5.7,1) node {$x_4$};
    \draw [color=black] (7.7,1) node {$x_3$};
    \draw [color=black] (9.7,1) node {$x_4$};
    \draw [color=black] (11.7,1) node {$x_3$};
    \draw [color=black] (13.7,1) node {$x_4$};
    \draw [color=black] (15.7,1) node {$x_3$};
    \draw [color=black] (17.7,1) node {$x_4$};
    \draw [color=black] (19.7,1) node {$x_3$};
    \draw [color=black] (21.7,1) node {$x_4$};
    \draw [color=black] (23.7,1) node {$x_3$};
    \draw [color=black] (25.7,1) node {$x_4$};

    \draw [fill=black] (0,0) circle (2pt);
    \draw[color=black] (0,-0.5) node {$B_{0}$};
    \draw [fill=black] (2,0) circle (2pt);
    \draw[color=black] (2,-0.5) node {$B_{1}$};
    \draw [fill=black] (4,0) circle (2pt);
    \draw[color=black] (4,-0.5) node {$B_{2}$};
    \draw [fill=black] (6,0) circle (2pt);
    \draw[color=black] (6,-0.5) node {$B_{3}$};
    \draw [fill=black] (8,0) circle (2pt);
    \draw[color=black] (8,-0.5) node {$B_{4}$};
    \draw [fill=black] (10,0) circle (2pt);
    \draw[color=black] (10,-0.5) node {$B_{5}$};
    \draw [fill=black] (12,0) circle (2pt);
    \draw[color=black] (12,-0.5) node {$B_{6}$};
    \draw [fill=black] (14,0) circle (2pt);
    \draw[color=black] (14,-0.5) node {$B_{7}$};
    \draw [fill=black] (16,0) circle (2pt);
    \draw[color=black] (16,-0.5) node {$B_{8}$};
    \draw [fill=black] (18,0) circle (2pt);
    \draw[color=black] (18,-0.5) node {$B_{9}$};
    \draw [fill=black] (20,0) circle (2pt);
    \draw[color=black] (20,-0.5) node {$B_{10}$};
    \draw [fill=black] (22,0) circle (2pt);
    \draw[color=black] (22,-0.5) node {$B_{11}$};
    \draw [fill=black] (24,0) circle (2pt);
    \draw[color=black] (24,-0.5) node {$B_{12}$};
    \draw [fill=black] (26,0) circle (2pt);
    \draw[color=black] (26,-0.5) node {$B_{13}$};

    \draw [fill=black] (0,2) circle (2pt);
    \draw[color=black] (0,2.5) node {$A_{0}$};
    \draw [fill=black] (2,2) circle (2pt);
    \draw[color=black] (2,2.5) node {$A_{1}$};
    \draw [fill=black] (4,2) circle (2pt);
    \draw[color=black] (4,2.5) node {$A_{2}$};
    \draw [fill=black] (6,2) circle (2pt);
    \draw[color=black] (6,2.5) node {$A_{3}$};
    \draw [fill=black] (8,2) circle (2pt);
    \draw[color=black] (8,2.5) node {$A_{4}$};
    \draw [fill=black] (10,2) circle (2pt);
    \draw[color=black] (10,2.5) node {$A_{5}$};
    \draw [fill=black] (12,2) circle (2pt);
    \draw[color=black] (12,2.5) node {$A_{6}$};
    \draw [fill=black] (14,2) circle (2pt);
    \draw[color=black] (14,2.5) node {$A_{7}$};
    \draw [fill=black] (16,2) circle (2pt);
    \draw[color=black] (16,2.5) node {$A_{8}$};
    \draw [fill=black] (18,2) circle (2pt);
    \draw[color=black] (18,2.5) node {$A_{9}$};
    \draw [fill=black] (20,2) circle (2pt);
    \draw[color=black] (20,2.5) node {$A_{10}$};
    \draw [fill=black] (22,2) circle (2pt);
    \draw[color=black] (22,2.5) node {$A_{11}$};
    \draw [fill=black] (24,2) circle (2pt);
    \draw[color=black] (24,2.5) node {$A_{12}$};
    \draw [fill=black] (26,2) circle (2pt);
    \draw[color=black] (26,2.5) node {$A_{13}$};
    
\end{scriptsize}

\begin{scope}[shift={(0,-4)}]
    \draw [line width=0.2pt] (0,0) -- (24,0);
    \draw [line width=0.2pt] (0,2) -- (24,2);
    \draw [line width=0.2pt] (0,0) -- (0,2);
    \draw [line width=0.2pt] (2,0) -- (2,2);
    \draw [line width=0.2pt] (4,0) -- (4,2);
    \draw [line width=0.2pt] (6,0) -- (6,2);
    \draw [line width=0.2pt] (8,0) -- (8,2);
    \draw [line width=0.2pt] (10,0) -- (10,2);
    \draw [line width=0.2pt] (12,0) -- (12,2);
    \draw [line width=0.2pt] (14,0) -- (14,2);
    \draw [line width=0.2pt] (16,0) -- (16,2);
    \draw [line width=0.2pt] (18,0) -- (18,2);
    \draw [line width=0.2pt] (20,0) -- (20,2);
    \draw [line width=0.2pt] (22,0) -- (22,2);
    \draw [line width=0.2pt] (24,0) -- (24,2);

    \draw [line width=2pt] (0,0) -- (2,0);
    \draw [line width=2pt] (0,2) -- (2,2);
    \draw [line width=2pt] (4,0) -- (4,2);
    \draw [line width=2pt] (6,0) -- (6,2);
    \draw [line width=2pt] (8,0) -- (10,0);
    \draw [line width=2pt] (8,2) -- (10,2);
    \draw [line width=2pt] (12,0) -- (12,2);
    \draw [line width=2pt] (14,0) -- (14,2);
    \draw [line width=2pt] (16,0) -- (16,2);
    \draw [line width=2pt] (18,0) -- (20,0);
    \draw [line width=2pt] (18,2) -- (20,2);
    \draw [line width=2pt] (22,0) -- (24,0);
    \draw [line width=2pt] (22,2) -- (24,2);
\end{scope}

\begin{scope}[shift={(0,-8)}]
    \draw [line width=0.2pt] (0,0) -- (26,0);
    \draw [line width=0.2pt] (0,2) -- (26,2);
    \draw [line width=0.2pt] (0,0) -- (0,2);
    \draw [line width=0.2pt] (2,0) -- (2,2);
    \draw [line width=0.2pt] (4,0) -- (4,2);
    \draw [line width=0.2pt] (6,0) -- (6,2);
    \draw [line width=0.2pt] (8,0) -- (8,2);
    \draw [line width=0.2pt] (10,0) -- (10,2);
    \draw [line width=0.2pt] (12,0) -- (12,2);
    \draw [line width=0.2pt] (14,0) -- (14,2);
    \draw [line width=0.2pt] (16,0) -- (16,2);
    \draw [line width=0.2pt] (18,0) -- (18,2);
    \draw [line width=0.2pt] (20,0) -- (20,2);
    \draw [line width=0.2pt] (22,0) -- (22,2);
    \draw [line width=0.2pt] (24,0) -- (24,2);
    \draw [line width=0.2pt] (26,0) -- (26,2);

    \draw [line width=2pt] (4,0) -- (6,0);
    \draw [line width=2pt] (4,2) -- (6,2);
    \draw [line width=2pt] (4,0) -- (4,2);
    \draw [line width=2pt] (6,0) -- (6,2);
    \draw [line width=2pt] (12,0) -- (14,0);
    \draw [line width=2pt] (12,2) -- (14,2);
    \draw [line width=2pt] (12,0) -- (12,2);
    \draw [line width=2pt] (14,0) -- (14,2);
    \draw [line width=2pt] (16,0) -- (16,2);
    \draw [line width=2pt] (18,0) -- (20,0);
    \draw [line width=2pt] (18,2) -- (20,2);
    \draw [line width=2pt] (26,0) -- (26,2);
    \draw [line width=2pt] (22,0) -- (24,0);
    \draw [line width=2pt] (22,2) -- (24,2);
    \draw [line width=2pt] (16,0) -- (18,0);
    \draw [line width=2pt] (16,2) -- (18,2);
    \draw [line width=2pt] (20,0) -- (22,0);
    \draw [line width=2pt] (20,2) -- (22,2);
    \draw [line width=2pt] (24,0) -- (26,0);
    \draw [line width=2pt] (24,2) -- (26,2);

    \begin{scriptsize}
        \draw (1,1) node {$S_1$};
        \draw (3,1) node {$S_2$};
        \draw (5,1) node {$S_1$};
        \draw (7,1) node {$S_2$};
        \draw (9,1) node {$S_1$};
        \draw (11,1) node {$S_2$};
        \draw (13,1) node {$S_1$};
        \draw (15,1) node {$S_2$};
        \draw (17,1) node {$S_1$};
        \draw (19,1) node {$S_2$};
        \draw (21,1) node {$S_1$};
        \draw (23,1) node {$S_2$};
        \draw (25,1) node {$S_1$};
    \end{scriptsize}
\end{scope}

\begin{scope}[shift={(0,-12)}]
    \draw [line width=0.2pt] (0,0) -- (24,0);
    \draw [line width=0.2pt] (0,2) -- (24,2);
    \draw [line width=0.2pt] (0,0) -- (0,2);
    \draw [line width=0.2pt] (2,0) -- (2,2);
    \draw [line width=0.2pt] (4,0) -- (4,2);
    \draw [line width=0.2pt] (6,0) -- (6,2);
    \draw [line width=0.2pt] (8,0) -- (8,2);
    \draw [line width=0.2pt] (10,0) -- (10,2);
    \draw [line width=0.2pt] (12,0) -- (12,2);
    \draw [line width=0.2pt] (14,0) -- (14,2);
    \draw [line width=0.2pt] (16,0) -- (16,2);
    \draw [line width=0.2pt] (18,0) -- (18,2);
    \draw [line width=0.2pt] (20,0) -- (20,2);
    \draw [line width=0.2pt] (22,0) -- (22,2);
    \draw [line width=0.2pt] (24,0) -- (24,2);

    \draw [line width=2pt] (4,0) -- (6,0);
    \draw [line width=2pt] (4,2) -- (6,2);
    \draw [line width=2pt] (4,0) -- (4,2);
    \draw [line width=2pt] (6,0) -- (6,2);
    \draw [line width=2pt] (12,0) -- (14,0);
    \draw [line width=2pt] (12,2) -- (14,2);
    \draw [line width=2pt] (12,0) -- (12,2);
    \draw [line width=2pt] (14,0) -- (14,2);
    \draw [line width=2pt] (16,0) -- (16,2);
    \draw [line width=2pt] (18,0) -- (20,0);
    \draw [line width=2pt] (18,2) -- (20,2);
    \draw [line width=2pt] (24,0) -- (24,2);
    \draw [line width=2pt] (22,0) -- (24,0);
    \draw [line width=2pt] (22,2) -- (24,2);
    \draw [line width=2pt] (16,0) -- (18,0);
    \draw [line width=2pt] (16,2) -- (18,2);
    \draw [line width=2pt] (20,0) -- (22,0);
    \draw [line width=2pt] (20,2) -- (22,2);

    \begin{scriptsize}
        \draw (1,1) node {$S_1$};
        \draw (3,1) node {$S_2$};
        \draw (5,1) node {$S_1$};
        \draw (7,1) node {$S_2$};
        \draw (9,1) node {$S_1$};
        \draw (11,1) node {$S_2$};
        \draw (13,1) node {$S_1$};
        \draw (15,1) node {$S_2$};
        \draw (17,1) node {$S_1$};
        \draw (19,1) node {$S_2$};
        \draw (21,1) node {$S_1$};
        \draw (23,1) node {$S_2$};
    \end{scriptsize}
\end{scope}

\begin{scriptsize}
    \draw (-2,1) node {$P$};
    \draw (-2,-3) node {$P_1$};
    \draw (-2,-7) node {$H_{P}$};
    \draw (-2,-11) node {$H_{P_1}$};  
\end{scriptsize}

\end{tikzpicture}
\caption{Perfect matchings and their symmetric differences. (The colors will be explained in Example \ref{ex: phi}.)}
\label{fig: perfect}
\end{figure}

\begin{proposition}
With the above notations, for $k \geq 1$,
\begin{equation}
    z_{2k+1}= (x_1^{-1} x_4y_1) z_{2k} +(x_1^{-1} x_2) z_{2k-1}.
\end{equation}
\begin{equation}
    z_{2k+2}= (x_2^{-1}x_3) z_{2k+1} + (x_1x_2^{-1}y_1y_2) z_{2k}.
\end{equation}
\end{proposition}

\begin{proof}
If $P$ is a perfect matching of $H_{2n+1}$ then $P$ contains $\{A_{2k+1}B_{2k+1}\}$ or $\{A_{2k}A_{2k+1},B_{2k}B_{2k+1}\}$ in the last portion.
\begin{itemize}
\item If the perfect matching $P$ contains $\{A_{2k+1}B_{2k+1}\}$ (see Figure \ref{fig: perfect}):
\begin{itemize}
    \item $P_1=P \setminus \{A_{2k+1}B_{2k+1}\}$ is a perfect matching of $H_{2n}$,
    \item $H_{P}=H_{P_1} \cup S_1$.
\end{itemize}

\item If the perfect matching $P$ contains $\{A_{2k}A_{2k+1},B_{2k}B_{2k+1}\}$ (similar to the previous case):
\begin{itemize}
    \item $P_2=P \setminus \{A_{2k}A_{2k+1},B_{2k}B_{2k+1}\}$ is a perfect matching of $H_{2n-1}$,
    \item $H_{P}=H_{P_2}$.
\end{itemize}

\end{itemize}
Conversely, we can get a perfect matching of $H_{2n+1}$ from perfect matching of $H_{2n}$ by adding the edge $A_{2k+1}B_{2k+1}$ and from perfect matching of $H_{2n-1}$ by adding the edges $A_{2k}A_{2k+1},B_{2k}B_{2k+1}$. So, from Theorem \ref{thm: y formula}--\ref{thm: x formula}, we have:
\[
\begin{aligned}
    z_{2k+1} & = \sum_{P \in H_{2k+1}} \frac{w(P)y(P)}{x_1^{k+1}x_2^k} \\
    & = \sum_{P_1 \in H_{2k}} \frac{(x_4w(P_1))(y_1y(P_1))}{x_1^{k+1}x_2^k} + \sum_{P_2 \in H_{2k-1}} 
    \frac{(x_2^2w(P_2))(y(P_2))}{x_1^{k+1}x_2^k}\\
    & = (x_1^{-1}x_4) z_{2k}+ (x_1^{-1}x_2) z_{2k-1}.
\end{aligned}
\]
By the same method for the perfect matchings of $H_{2n+2}$, we get
\[z_{2k+2}= (x_2^{-1}x_3) z_{2k+1} + (x_1x_2^{-1}y_1y_2) z_{2k} . \]
\end{proof}

\begin{corollary}
The number of perfect matchings on the graph $H_n$ is $F_{n+2}$.
\end{corollary}

\begin{proof}
Let $|z_k|$ be the value of $z_k$ by setting $x_1=x_2=x_3=x_4=y_1=y_2=1$. We will prove by induction that $$|z_k|=F_{k+2}$$ for any $k \geq 1$. The base case $|z_1|=2=F_3$ and $|z_2|=3=F_4$ are true as shown in Example \ref{ex: z123}. Suppose that the claim is true for $k\leq n$,
\begin{itemize}
    \item If $n$ is odd,
    \[z_{n+1}= (x_1^{-1} x_4y_1) z_{n} +(x_1^{-1} x_2) z_{n-1} \implies |z_{n+1}|=|z_{n}|+|z_{n-1}|=F_{n+2}+F_{n+1}=F_{n+3}. \]
    \item If $n$ is even,
    \[z_{n+1}= (x_2^{-1} x_3) z_{n} +(x_1 x_2^{-1} y_1y_2) z_{n-1} \implies |z_{n+1}|=|z_{n}|+|z_{n-1}|=F_{n+2}+F_{n+1}=F_{n+3}. \]
\end{itemize}
Thus the claim is true for $k=n+1$, so the claim is true for any $k \geq 1$ by induction. Moreover, by the definition
\[z_n= \sum_{P} \frac{w(P) y(P)}{x_{i_1}x_{i_2}...x_{i_d}}\]
where the sum is over all perfect matchings $P$ of $H_n$. Since each term becomes $1$ under the evaluation, $|z_n|$ is the number of perfect matchings on the graph $H_n$.
\end{proof}

\begin{corollary}\label{clusFib}
In the coefficient free case, the sum of coefficients in the Laurent polynomials of the cluster variables $x_{\gamma_k}$ is $F_{2n-3}$
\end{corollary}

\begin{proof}
The sum of coefficients in the Laurent polynomials is equal to the number of perfect matchings on the graph $G_{\gamma_k}=H_{2k-5}$, that is $F_{2n-3}$.
\end{proof}

\section{Compatible pairs on the maximal Dyck path}
\label{sec4}

In this section, we prove that the cluster variable $x_n$ in $\mathcal{A}(2,2)$ is the greedy element $x[n-2,n-3]$ which is a Laurent polynomial of the initial cluster $\{x_1,x_2\}$. Then we use compatible pairs on the maximal Dyck path $\mathcal{D}^{(n-2) \times (n-3)}$ to calculate these greedy elements, see Section \ref{sec2.4} for the terminologies.\\

Consider the coefficient free cluster algebra $\mathcal{A}(2,2)$ with initial cluster $\{x_1,x_2\}$ and exchange relation $x_{n-1}x_{n+1}=x_n^2+1$ for $n\in\mathbb{Z}$. We first prove the following lemma.

\begin{lemma}
For $n\geq 2$, the cluster variable
\begin{equation}
    x_n=x[n-2,n-3].
\end{equation}
\end{lemma}

\begin{proof}
By Proposition \ref{prop: sigma formula},
\[ \sigma_1(x[a_1,a_2])=x[a_1, 2[a_1]_{+} -a_2] \text{ and } \sigma_2(x[a_1,a_2])=x[2[a_2]_{+} -a_1,a_2].\]
For the base cases, $x_1= x_1^1x_2^0=x[-1,0]$ and $x_2=x_1^0x_2^1=x[0,-1]$, thus
\[x_3=\sigma_2(x_1)=\sigma_2(x[-1,0])=x[1,0].\]
Suppose that $x_k=x[k-2,k-3]$ for any $k \leq n+1$, then:
\[
\begin{aligned}
x_{n+2} & = \sigma_2 (\sigma_1 (x_n))\\
& = \sigma_2 (\sigma_1(x[n-2,n-3]))\\
&= \sigma_2(x[n-2,n-1 ])\\
&= x[n,n-1].
\end{aligned}
\]
Hence the statement is true by induction.
\end{proof}

The maximal Dyck path $\mathcal{D}^{(n+1) \times n}$ is determined as in Figure \ref{fig: compatible}, where we label its vertices by $A_0, A_1,...A_n, A_{n+1}\equiv A_0$, and $B_0 \equiv A_1, B_1,B_2,...,B_n$. Let $\mathcal{D}_1=\{u_0,u_1,...,u_n\}, \mathcal{D}_2=\{v_1,v_2,...,v_n\}$ be the sets of horizontal and vertical edges respectively on the maximal Dyck path. Note that
\Eq{
u_i&=A_iB_i,\quad\quad 0\leq i\leq n,\\
v_j&=B_jA_{j+1},\quad\quad 1\leq j\leq n.
}
For $S_1 \subset \mathcal{D}_1, S_2 \subset \mathcal{D}_2$, to facilitate the calculation we define the following functions.

\begin{definition}
For $C \in (AB)^{\circ}$, the \emph{step functions} are defined by
\Eq{
    f(C,AB) &:= |(AC)_2| - 2|(AC)_1 \cap S_1|,\\
    g(C,AB) &:= |(CB)_1| - 2|(CB)_2 \cap S_2|.
}
\end{definition}
\begin{proposition} \label{prop: max}
In the maximal Dyck path $\mathcal{D}^{(n+1) \times n}$, $(S_1,S_2)$ is a compatible pair if and only if for any $u_i \in S_1, v_j \in S_2$:
\begin{equation}\label{maxeq}
    \max_{C \in (A_iA_{j+1})^{\circ}} \{f(C,A_iA_{j+1}), g(C,A_iA_{j+1})\} \geq 0.
\end{equation}
\end{proposition}

\begin{proof}
By Definition \ref{compatibledef}, $(S_1,S_2)$ is a compatible pair if and only if there exists $C \in (A_iA_{j+1})^{\circ}$ such that $f(C, A_iA_{j+1})=0$ or $g(C, A_iA_{j+1})=0$. 

In particular if $(S_1,S_2)$ is a compatible pair, then condition \eqref{maxeq} clearly holds.

Now assume the condition \eqref{maxeq} holds. We observe that if $u_i \in S_1, v_j \in S_2$ then $$f(B_i,A_iA_{j+1})=g(B_j,A_iA_{j+1})=-2$$ when $C\in(A_iA_{j+1})^\circ$ is nearest to $A_i$ or $A_{j+1}$. We also note that
\begin{itemize}
    \item if $A_k,B_k \in (A_iA_{j+1})^{\circ}$ where $0 \leq k \leq n$, then 
    \[f(B_k, A_iA_{j+1})-f(A_k, A_iA_{j+1}) \in \{0,-2\}, \]
    \item if $B_k,A_{k+1} \in (A_iA_{j+1})^{\circ}$ where $1 \leq k \leq n$, then
    \[f(A_{k+1}, A_iA_{j+1})-f(B_k, A_iA_{j+1})=1. \]
\end{itemize}
Thus, $f$ may increase by no more than 1 when $C$ travels from $B_i$ to $B_j$ in the Northeast direction. Similarly, 
\begin{itemize}
    \item if $A_k,B_k \in (A_iA_{j+1})^{\circ}$ for some $0 \leq k \leq n$, then 
    \[g(A_k, A_iA_{j+1})-g(B_k, A_iA_{j+1})=1, \]
    \item if $B_k,A_{k+1} \in (A_iA_{j+1})^{\circ}$ for some $1 \leq k \leq n$, then
    \[g(B_k, A_iA_{j+1})-g(A_{k+1}, A_iA_{j+1}) \in \{0,-2\}. \]
\end{itemize}
Thus, $g$ may increase by no more than 1 when $C$ travels from $B_j$ back to $B_i$ in the Southwest direction. 

Since $f$ may increase by no more than 1, if $f(C, A_iA_{j+1})>0$, there exists $D \in (A_iC)^{\circ}$ such that $f(D, A_iA_{j+1})=0$. Similarly since $g$ may decrease by no more than $1$ when $C$ travels backward, if $g(C, A_iA_{j+1})>0$ for some $C \in (A_iA_{j+1})^{\circ}$, there exists $D \in (CA_{j+1})^{\circ}$ such that $g(D, A_iA_{j+1})=0$. In other words $(S_1, S_2)$ is a compatible pair. 
\end{proof}

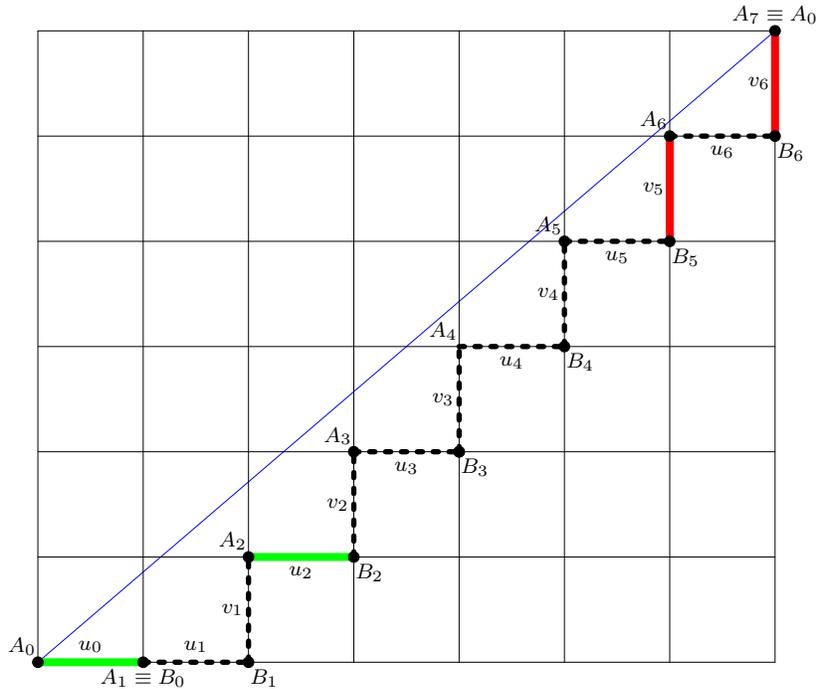
\begin{figure}[H]
\centering
\begin{tikzpicture}[line cap=round,line join=round,>=triangle 45,x=0.7cm,y=0.7cm]
    \draw[line width=0.2pt][color=blue] (0,0) -- (14,12);

    \draw [line width=0.2pt] (0,0) -- (14,0);
    \draw [line width=0.2pt] (0,2) -- (14,2);
    \draw [line width=0.2pt] (0,4) -- (14,4);
    \draw [line width=0.2pt] (0,6) -- (14,6);
    \draw [line width=0.2pt] (0,8) -- (14,8);
    \draw [line width=0.2pt] (0,10) -- (14,10);
    \draw [line width=0.2pt] (0,12) -- (14,12);
    \draw [line width=0.2pt] (0,0) -- (0,12);
    \draw [line width=0.2pt] (2,0) -- (2,12);
    \draw [line width=0.2pt] (4,0) -- (4,12);
    \draw [line width=0.2pt] (6,0) -- (6,12);
    \draw [line width=0.2pt] (8,0) -- (8,12);
    \draw [line width=0.2pt] (10,0) -- (10,12);
    \draw [line width=0.2pt] (12,0) -- (12,12);
    \draw [line width=0.2pt] (14,0) -- (14,12);

    \draw [line width=3pt, color=green] (0,0) -- (2,0);
    \draw [line width=2pt, loosely dotted] (2,0) -- (4,0);
    \draw [line width=2pt, loosely dotted] (4,0) -- (4,2);
    \draw [line width=3pt, color=green] (4,2) -- (6,2);
    \draw [line width=2pt, loosely dotted] (6,2) -- (6,4);
    \draw [line width=2pt, loosely dotted] (6,4) -- (8,4);
    \draw [line width=2pt, loosely dotted] (8,4) -- (8,6);
    \draw [line width=2pt, loosely dotted] (8,6) -- (10,6);
    \draw [line width=2pt, loosely dotted] (10,6) -- (10,8);
    \draw [line width=2pt, loosely dotted] (10,8) -- (12,8);
    \draw [line width=3pt, color=red] (12,8) -- (12,10);
    \draw [line width=2pt, loosely dotted] (12,10) -- (14,10);
    \draw [line width=3pt, color=red] (14,10) -- (14,12);

\begin{scriptsize}
    \draw [color=black] (1,0.3) node {$u_{0}$};
    \draw [color=black] (3,0.3) node {$u_{1}$};
    \draw [color=black] (5,1.7) node {$u_{2}$};
    \draw [color=black] (7,3.7) node {$u_{3}$};
    \draw [color=black] (9,5.7) node {$u_{4}$};
    \draw [color=black] (11,7.7) node {$u_{5}$};
    \draw [color=black] (13,9.7) node {$u_{6}$};

    \draw [color=black] (3.7,1) node {$v_{1}$};
    \draw [color=black] (5.7,3) node {$v_{2}$};
    \draw [color=black] (7.7,5) node {$v_{3}$};
    \draw [color=black] (9.7,7) node {$v_{4}$};
    \draw [color=black] (11.7,9) node {$v_{5}$};
    \draw [color=black] (13.7,11) node {$v_{6}$};

    \draw [fill=black] (0,0) circle (2pt);
    \draw[color=black] (-0.3,0.3) node {$A_{0}$};
    \draw [fill=black] (2,0) circle (2pt);
    \draw[color=black] (2,-0.3) node {$A_{1} \equiv B_0$};
    \draw [fill=black] (4,2) circle (2pt);
    \draw[color=black] (3.7,2.3) node {$A_{2}$};
    \draw [fill=black] (6,4) circle (2pt);
    \draw[color=black] (5.7,4.3) node {$A_{3}$};
    \draw [fill=black] (8,4) circle (2pt);
    \draw[color=black] (7.7,6.3) node {$A_{4}$};
    \draw [fill=black] (10,8) circle (2pt);
    \draw[color=black] (9.7,8.3) node {$A_{5}$};
    \draw [fill=black] (12,10) circle (2pt);
    \draw[color=black] (11.7,10.3) node {$A_{6}$};
    \draw [fill=black] (14,12) circle (2pt);
    \draw[color=black] (14,12.3) node {$A_{7} \equiv A_{0}$};

    \draw [fill=black] (4,0) circle (2pt);
    \draw[color=black] (4.3,-0.3) node {$B_{1}$};
    \draw [fill=black] (6,2) circle (2pt);
    \draw[color=black] (6.3,1.7) node {$B_{2}$};
    \draw [fill=black] (8,4) circle (2pt);
    \draw[color=black] (8.3,3.7) node {$B_{3}$};
    \draw [fill=black] (10,6) circle (2pt);
    \draw[color=black] (10.3,5.7) node {$B_{4}$};
    \draw [fill=black] (12,8) circle (2pt);
    \draw[color=black] (12.3,7.7) node {$B_{5}$};
    \draw [fill=black] (14,10) circle (2pt);
    \draw[color=black] (14.3,9.7) node {$B_{6}$};

\end{scriptsize}
\end{tikzpicture}
\caption{A compatible pair on the maximal Dyck path $\mathcal{D}^{7 \times 6}$ with solid edges. (The colors will be explained in Example \ref{ex: phi}.) }
\label{fig: compatible}
\end{figure}

\begin{lemma} \label{lem: j-i}
$(S_1,S_2)$ is a compatible pair in $\mathcal{D}^{(n+1) \times n}$ if and only if
\begin{equation} \label{eq: empty}
    \{j-i|u_i \in S_1, v_j \in S_2 \} \cap \{0,1\} = \emptyset.
\end{equation}
\end{lemma}

\begin{proof}
One direction is clear:
\begin{itemize}
    \item if $u_i \in S_1,v_i \in S_2$ for some $1 \leq i \leq n$, then $(A_iA_{i+1})^{\circ}=\{B_i\}$. However,
    \[f(B_i,A_iA_{i+1})=g(B_i,A_iA_{i+1})=-2,\]
    so $(S_1,S_2)$ is not compatible by Proposition \ref{prop: max};
    \item if $u_i \in S_1, v_{i+1} \in S_2$ for some $0 \leq i \leq n-1$ then for any $C \in (A_iA_{i+2})^{\circ}=\{B_i,A_{i+1},B_{i+1}\}$,
    \[f(C,A_iA_{i+2})<0 \text{ and } g(C,A_iA_{i+2})<0\]
    so $(S_1,S_2)$ is not compatible again by Proposition \ref{prop: max}.
\end{itemize}
Conversely, if 
\[\{j-i|u_i \in S_1, v_j \in S_2 \} \cap \{0,1\} = \emptyset \]
then for any $u_i \in S_1,v_j \in S_2$, either $j-i>1$ or $i>j$. Suppose on the contrary that both
\[f(A_j,A_iA_{j+1}) \leq -1 \text{ and } g(A_{i+1},A_iA_{j+1}) \leq -1.\]

\textbf{Case 1:} Let $u_i\in S_1$ and $v_j\in S_2$ with $j-i>1$. Then we have
\[
\begin{aligned}
    |(A_iA_j)_1 \cap S_1|
    & = \frac{1}{2} \bigg( |(A_iA_j)_2| - f(A_j,A_iA_{j+1})  \bigg)\\
    & \geq \frac{1}{2} \bigg( (j-i-1) - (-1) \bigg)\\
    & = \frac{1}{2}(j-i),
\end{aligned}
\]
(note that $|(A_iA_j)_2|=j-i-1$ when $i=0$) and
\[
\begin{aligned}
    |(A_{i+1}A_{j+1})_2 \cap S_2|
    & = \frac{1}{2} \bigg( |(A_{i+1}A_{j+1})_1| - g(A_{i+1},A_iA_{j+1})  \bigg)\\
    & \geq \frac{1}{2} \bigg( (j-i) - (-1) \bigg)\\
    & = \frac{1}{2}(j-i+1).
\end{aligned}
\]
Thus
\[|A_iA_{j+1} \cap (S_1 \cup S_2)| \geq |(A_iA_j)_1 \cap S_1|+ |(A_{i+1}A_{j+1})_2 \cap S_2| \geq j-i + \frac{1}{2}\]
which implies
\[|A_iA_{j+1} \cap (S_1 \cup S_2)| \geq j-i+1.\]
On the other hand, $|A_kA_{k+1} \cap (S_1 \cup S_2)| \leq 1$ because of \eqref{eq: empty} for any $i \leq k \leq j$, so that
\[|A_iA_{j+1} \cap (S_1 \cup S_2)| \leq j-i+1.\]
Therefore $|A_iA_{j+1} \cap (S_1 \cup S_2)| = j-i+1$ and we must require $|A_kA_{k+1} \cap (S_1 \cup S_2)| = 1$ for any $i \leq k \leq j$ for this to hold.

However, if we take the smallest index $k \in [i+1,j]$ such that $v_{k} \in S_2$, which exists by assumption, then $u_{k-1} \notin S_1$ and $v_{k} \notin S_2$, which implies $|A_{k-1}A_{k} \cap (S_1 \cup S_2)|=0$, a contradiction.\\

\textbf{Case 2:} Let $u_i\in S_1$ and $v_j\in S_2$ with $i>j$. Then similarly we have

\[|(A_iA_j)_1 \cap S_1| \geq \frac{1}{2}(n+j-i+1)\]
and
\[|(A_{i+1}A_{j+1})_2 \cap S_2| \geq \frac{1}{2}(n+j-i+2).\]
(Recall that by convention we loop back from $(n+1,n)$ to $(0,0)$ when we consider the subpaths in this case.)

Therefore
\[|A_iA_{j+1} \cap (S_1 \cup S_2)| \geq n+j-i+2 .\]
On the other hand, by \eqref{eq: empty} again, $|A_kA_{k+1} \cap (S_1 \cup S_2)| \leq 1 $ for any $k \geq i$ or $k \leq j$, so
\[|A_iA_{j+1} \cap (S_1 \cup S_2)| \leq n+j-i+2 .\]
Therefore $|A_iA_{j+1} \cap (S_1 \cup S_2)| = n+j-i+2$ and we must again require $|A_kA_{k+1} \cap (S_1 \cup S_2)| = 1$ for any $k \geq i$ or $k \leq j$ for this to hold. In particular, $$|A_0A_2 \cap (S_1 \cup S_2)|=2$$ which implies $u_0,u_1 \in S_1$ and $v_1\notin S_2$ by \eqref{eq: empty}. Then taking $u_0\in S_1$ and $v_j\in S_2$ with $j\geq 2$ we reduce to Case 1.

Therefore we conclude that for any $u_i \in S_1,v_j \in S_2$:
\[f(A_j,A_iA_{j+1}) \geq 0 \quad\text{ or }\quad g(A_{i+1},A_iA_{j+1}) \geq 0\]
which implies $(S_1, S_2)$ is a compatible pair by Proposition \ref{prop: max}.
\end{proof}

Finally, we recall a well-known fact that
\begin{lemma}\label{oddFib} 
The odd-indexed Fibonacci sequence satisfies the recurrence relation
\Eq{F_{2n+1}=3F_{2n-1}-F_{2n-3},\quad n>1}
with initial terms $F_1=1, F_3=2, F_5=5$.
\end{lemma}
Let $T_n$ be the set of all compatible pairs on $\mathcal{D}^{(n+1) \times n}$.
\begin{theorem}\label{compFib}
We have the following recurrence relation for $n>1$:
\begin{equation}
    |T_n|=3|T_{n-1}|-|T_{n-2}|.
\end{equation}
In particular the number of compatible pairs in $\mathcal{D}^{(n+1) \times n}$ is $F_{2n+3}$.
\end{theorem}

\begin{proof}
Define the following subsets of $T_n$ :
\[O_n := \{(S_1,S_2)| u_n \notin S_1, v_n \notin S_2 \},\]
\[U_n := \{(S_1,S_2)| u_n \in S_1, v_n \notin S_2 \},\]
\[V_n := \{(S_1,S_2)| u_n \notin S_1, v_n \in S_2 \}.\]
Note that if $u_n \in S_1, v_n \in S_2$, then $(S_1,S_2)$ cannot be a compatible pair. Hence $$T_n = O_n \sqcup U_n \sqcup V_n.$$ From Lemma \ref{lem: j-i}, we see that $|O_n|=|U_n|=|T_{n-1}|$ and $|V_n|=|O_{n-1}|+|V_{n-1}|$. Applying the two relations above, we get:
\[
\begin{aligned}
|T_n| & = |O_n|+|U_n|+|V_n|\\
&=|T_{n-1}|+|T_{n-1}|+(|O_{n-1}|+|V_{n-1}|)\\
&=2|T_{n-1}|+(|T_{n-1}|-|U_{n-1}|)\\
&=3|T_{n-1}|-|T_{n-2}|.
\end{aligned}
\]
Since the initial terms $|T_0|=F_3=2$ and $|T_1|=F_5=5$, by Lemma \ref{oddFib} the theorem is proven.
\end{proof}

\section{Nondecreasing Dyck paths of length $2n$}
\label{sec5}

In this section, we will count the number of the so-called \emph{nondecreasing Dyck paths} of length $2n$. For further details, the reader is referred to \cite{deutsch2003bijection}.\\

\begin{definition}
A \emph{Dyck path of length $2n$} is a discrete path consisting of diagonal segments going from $(0,0)$ to $(2n,0)$ in the first quadrant $\{(x,y): x,y\in\mathbb{Z}_{\geq 0}\}$, such that each step equals to a vector $\vec{u}:=(1,1)$ or $\vec{v}:=(1,-1)$. 
\end{definition}

In particular a Dyck path $w$ of length $2n$ is equivalent to a sequence of $2n$ vectors $\vec{w}_1, \vec{w}_2,...,\vec{w}_{2n}$ where $\vec{w}_i$ equals $\vec{u}$ or $\vec{v}$.

\begin{remark} \label{rem: 2 dyck}
The notion of \enquote{Dyck path} here has a different meaning compared to that in Section \ref{sec2.4}. However, they are the same if we rotate the path defined here counterclockwise by 45 degrees. Specifically, a Dyck path of length $2n$ is equivalent to a normal Dyck path on $\mathcal{D}^{n \times n}$ in Section \ref{sec2.4}.
\end{remark}

\begin{definition}
Suppose $(x,y)=\vec{w}_1+...+\vec{w}_i$ for some $1 \leq i < 2n$ is a point on the Dyck path $w$. It is a \emph{peak} if $\vec{w}_i=\vec{u}$ and $\vec{w}_{i+1}=\vec{v}$. Conversely, it is a \emph{valley} if $\vec{w}_i=\vec{v}$ and $\vec{w}_{i+1}=\vec{u}$.
\end{definition}

\begin{definition}
 We call $y$ the \emph{altitude} of a point $(x,y)$. A Dyck path of length $2n$ is \emph{nondecreasing} if, viewed horizontally, the sequence of altitudes of the valleys is nondecreasing.
\end{definition}

\begin{definition}
A \emph{mountain} of a nondecreasing Dyck path $w$ of length $2n$ is the subpath from a valley to the next valley on the right, or from the last valley to the endpoint $(2n,0)$. The \emph{magnitude} of a mountain is a pair of positive integers $(d,e)$ such that $d,e\geq 1$ are the numbers of vectors $\vec{u}$ and $\vec{v}$ in the mountain's representation. 
\end{definition}

\begin{proposition}
A nondecreasing Dyck path can be viewed as a consecutive of $k$ mountains with mountain magnitudes $(d_i,e_i)$ where $d_i \geq e_i$ respectively, except the last mountain where we have $d_k \leq e_k$.
\end{proposition}

\begin{proof}
 The altitudes of the valleys are nondecreasing, implying $d_i \geq e_i$. The altitude of the last valley is nonnegative so $d_k \leq e_k$.
\end{proof}

\begin{example}
Figure \ref{fig: dyck} shows an example of a nondecreasing Dyck path whose length of 16 corresponds to consecutive mountains
\begin{equation}
    (d_1,e_1)=(2,2), \quad (d_2,e_2)=(2,2), \quad(d_3,e_3)=(3,1), \quad (d_4,e_4)=(1,3).
\end{equation}
\end{example}

Let $S_n$ be the set of all nondecreasing Dyck paths of length $2n$.

\begin{definition}
We define a disjoint partition $A_n, B_n, C_n$ of $S_n$ as follows:
\begin{itemize}
    \item paths in $A_n$ have the last mountain in the form $(d_k,e_k)$ with $d_k \geq 2$,
    \item paths in $B_n$ have the last two mountain in the form $(d_{k-1},e_{k-1}),(d_k,e_k)$ with $d_{k-1}=e_{k-1}$ and $d_k=1$,
    \item paths in $C_n$ have the last two mountain in the form $(d_{k-1},e_{k-1}),(d_k,e_k)$ with $d_{k-1}>e_{k-1}$ and $d_k=1$.
\end{itemize}
\end{definition}

\begin{figure}[H]
\centering
\begin{tikzpicture}[line cap=round,line join=round,>=triangle 45,x=0.8cm,y=0.8cm]
    \draw [line width=0.2pt] (0,0) -- (16,0);
    \draw [line width=0.2pt] (0,1) -- (16,1);
    \draw [line width=0.2pt] (0,2) -- (16,2);
    \draw [line width=0.2pt] (0,3) -- (16,3);
    \draw [line width=0.2pt] (0,4) -- (16,4);
    \draw [line width=0.2pt] (0,5) -- (16,5);
    \draw [line width=0.2pt] (0,6) -- (16,6);

    \draw [line width=0.2pt] (0,0) -- (0,6);
    \draw [line width=0.2pt] (1,0) -- (1,6);
    \draw [line width=0.2pt] (2,0) -- (2,6);
    \draw [line width=0.2pt] (3,0) -- (3,6);
    \draw [line width=0.2pt] (4,0) -- (4,6);
    \draw [line width=0.2pt] (5,0) -- (5,6);
    \draw [line width=0.2pt] (6,0) -- (6,6);
    \draw [line width=0.2pt] (7,0) -- (7,6);
    \draw [line width=0.2pt] (8,0) -- (8,6);
    \draw [line width=0.2pt] (9,0) -- (9,6);
    \draw [line width=0.2pt] (10,0) -- (10,6);
    \draw [line width=0.2pt] (11,0) -- (11,6);
    \draw [line width=0.2pt] (12,0) -- (12,6);
    \draw [line width=0.2pt] (13,0) -- (13,6);
    \draw [line width=0.2pt] (14,0) -- (14,6);
    \draw [line width=0.2pt] (15,0) -- (15,6);
    \draw [line width=0.2pt] (16,0) -- (16,6);

    \draw [line width=2pt] (0,0) -- (2,2);
    \draw [line width=2pt] (2,2) -- (4,0);
    \draw [line width=2pt] (4,0) -- (6,2);
    \draw [line width=2pt] (6,2) -- (8,0);
    \draw [line width=2pt] (8,0) -- (11,3);
    \draw [line width=2pt] (11,3) -- (12,2);
    \draw [line width=2pt] (12,2) -- (13,3);
    \draw [line width=2pt] (13,3) -- (16,0);

\begin{scriptsize}
    \draw[color=black] (0.7,1.3) node {$d_{1}$};
    \draw[color=black] (3.3,1.3) node {$e_{1}$};
    \draw[color=black] (4.7,1.3) node {$d_{2}$};
    \draw[color=black] (7.3,1.3) node {$e_{2}$};
    \draw[color=black] (9.2,1.8) node {$d_{3}$};
    \draw[color=black] (11.8,2.8) node {$e_{3}$};

\end{scriptsize}
\end{tikzpicture}
\caption{A nondecreasing Dyck path of length 16}
\label{fig: dyck}
\end{figure}
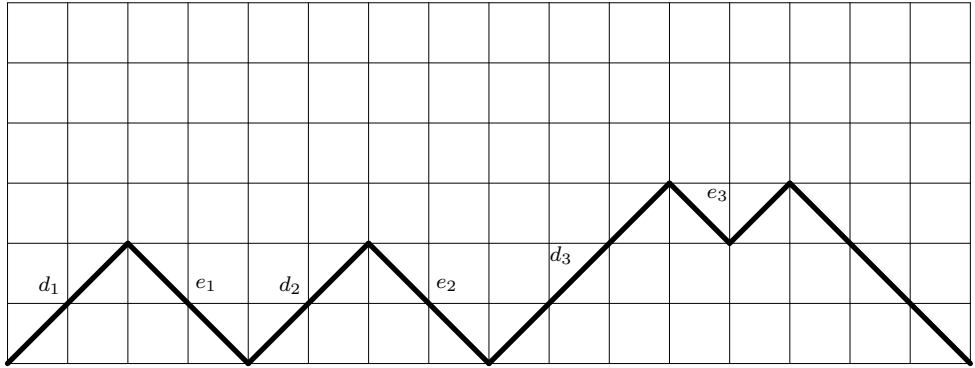

\begin{lemma} \label{lem: ABCrelations}
We have the following recurrence relations:
\Eq{
    |A_n|&=|S_{n-1}|,\\
    |B_n|&=|S_{n-1}|,\\
    |C_n|&= |B_{n-1}|+|C_{n-1}|.
}
\end{lemma}

\begin{proof}
Consider the map $A_n \to S_{n-1}$ defined by
\begin{equation}
    \bigg((d_1,e_1),...,(d_k,e_k)\bigg) \mapsto \bigg((d_1,e_1),...,(d_k-1,e_k-1)\bigg).
\end{equation}
It is clearly injective, with the inverse given by
\[\bigg((d_1,e_1),...,(d_k,e_k)\bigg) \mapsto \bigg((d_1,e_1),...,(d_k +1,e_k+1)\bigg)\]
which is well-defined because $d_k+1 \geq 2$.\\

Next, the map $B_n \to S_{n-1}$ defined by
\begin{equation}
    \bigg((d_1,e_1),...,(d_{k-1},e_{k-1}),(1,e_k)\bigg) \mapsto \bigg((d_1,e_1),...,(d_{k-1},d_{k-1}+e_k-1)\bigg)
\end{equation}
is clearly injective with the inverse given by
\[ \bigg((d_1,e_1),...,(d_{k-1},e_{k-1})\bigg) \mapsto \bigg((d_1,e_1),...,(d_{k-1},d_{k-1}),(1,e_{k-1}-d_{k-1}+1)\bigg) \]
which is well-defined because $(d_{k-1},e_{k-1})$ is the last mountain, so $e_{k-1} \geq d_{k-1}$ and
\[e_{k-1}-d_{k-1}+1 \geq 1.\]

Finally consider the map $C_n \to B_{n-1} \sqcup C_{n-1}$ defined by
\begin{equation}
    \bigg((d_1,e_1),...,(d_{k-1},e_{k-1}),(1,e_k)\bigg) \mapsto \bigg((d_1,e_1),...,(d_{k-1}-1,e_{k-1}),(1,e_k-1)\bigg).
\end{equation}
Since $d_{k-1}> e_{k-1} \implies d_{k-1}-1 \geq e_{k-1}$, it is clearly well-defined and injective. The inverse is given by
\[\bigg((d_1,e_1),...,(d_{k-1},e_{k-1}),(1,e_k)\bigg) \mapsto \bigg((d_1,e_1),...,(d_{k-1}+1,e_{k-1}),(1,e_k+1)\bigg),\]
which is well-defined because $d_{k-1}\geq e_{k-1}\implies d_{k-1}+1>e_{k-1}$. 
\end{proof}

\begin{theorem}\label{nonDyckFib}
We have the following recurrence relation for $n>1$:
\begin{equation}
    |S_{n+1}|=3|S_n|-|S_{n-1}|.
\end{equation}
In particular the number of nondecreasing Dyck paths of length $2n$ is $F_{2n-1}$.
\end{theorem}

\begin{proof}
Applying Lemma \ref{lem: ABCrelations}, we have:
\[
\begin{aligned}
|S_{n+1}| & = |A_{n+1}|+|B_{n+1}|+|C_{n+1}|\\
&= |S_n|+|S_n| + (|B_n|+|C_n|)\\
&= 2|S_n| +(|S_n| - |A_n|)\\
&= 3|S_n| -|S_{n-1}|.
\end{aligned}
\]
We also have the initial conditions $|S_2|=F_3=2$ and $|S_3|=F_5=5$, hence by Lemma \ref{oddFib} $|S_n|=F_{2n-1}$ as required.
\end{proof}

\section{Bijective correspondences}
\label{sec6}

In this section, we construct explicitly the one-to-one correspondences in both directions among three combinatorial objects shown in Figure \ref{fig: main} including
\begin{itemize}
    \item perfect matchings on the snake graph $G_{T, \gamma_{n+3}}$ of annulus with 2 marked points,
    \item compatible pairs on the maximal Dyck path $\mathcal{D}^{(n+1) \times n}$,
    \item nondecreasing Dyck paths with a length of $2n+4$.
\end{itemize}

\subsection{Perfect matchings and compatible pairs}
\label{sec6.1}

Using the notation from Figure \ref{fig: perfect} and Figure \ref{fig: compatible}, define
\Eq{\label{pertocom}
    \phi: \{\text{perfect matching on } G_{T, \gamma_{n+3}}\} &\to \{\text{compatible pairs on } \mathcal{D}^{(n+1) \times n}\}\\
    P&\mapsto (S_1,S_2)\nonumber
}
by the following rules:
\begin{itemize}
    \item $u_0 \in S_1 \iff A_0A_1,B_0B_1 \in P$
    \item $u_i \in S_1 \iff A_{2i}A_{2i+1},B_{2i}B_{2i+1} \in P, \forall i \in \{1,2,...,n\}$
    \item $v_i \in S_2 \iff A_{2i-1}A_{2i},B_{2i-1}B_{2i} \in P, \forall i \in \{1,2,...,n\}$
\end{itemize}

\begin{example} \label{ex: phi} When $n=6$, consider the following perfect matching on $G_{T, \gamma_{9}}$,
\begin{equation}
    \begin{aligned}
P= & \{A_0A_1,B_0B_1,A_2B_2,A_3B_3,A_4A_5,B_4B_5,A_6B_6,\\
&A_7B_7,A_8B_8,A_9A_{10},B_9B_{10},A_{11}A_{12},B_{11}B_{12},A_{13}B_{13}\}
\end{aligned}
\end{equation}
(see Figure \ref{fig: perfect}). It responds to the compatible pair $(S_1,S_2)$ on $\mathcal{D}^{7 \times 6}$ where
\begin{equation}
    S_1= \{u_0,u_2\}, \quad\quad S_2=\{v_5,v_6\}
\end{equation}
(see Figure \ref{fig: compatible}) with the correspondence highlighted with the same colors on both figures.
\end{example}

\begin{theorem} \label{thm: phi}
The map $\phi$ gives a one-to-one correspondence between perfect matchings on $G_{T,\gamma_{n+3}}$ and compatible pairs on $\mathcal{D}^{(n+1)\times n}$.
\end{theorem}
To prove the theorem, we first give the following lemma.

\begin{lemma} 
The map $\phi$ is well-defined.
\end{lemma}

\begin{proof}
If $u_i \in S_1$ then $A_{2i}A_{2i+1},B_{2i}B_{2i+1} \in P$. Thus by definition of a perfect matching,
\begin{itemize}
    \item $A_{2i-1}A_{2i},B_{2i-1}B_{2i} \notin P \implies v_i \notin S_2$,
    \item $A_{2i+1}A_{2i+2},B_{2i+1}B_{2i+2} \notin P \implies v_{i+1} \notin S_2$.
\end{itemize}
Therefore,
\[\{j-i|u_i \in S_1, v_j \in S_2 \} \cap \{0,1\} = \emptyset \]
which implies that $(S_1,S_2)$ is a compatible pair on $\mathcal{D}^{(n+1) \times n}$ by Lemma \ref{lem: j-i}.
\end{proof}

Note that the rules of $\phi$ are defined on both directions, so we also obtain the inverse map $\phi^{-1}$ by construction. Note that the vertical edges on $P$ are uniquely determined by the horizontal edges.

\begin{example}
We consider again Example \ref{ex: phi}, with the horizontal edges given by
\begin{equation}
    \{A_0A_1,B_0B_1,A_4A_5,B_4B_5,A_9A_{10},B_9B_{10},A_{11}A_{12},B_{11}B_{12}\} \subset P
\end{equation}
corresponding to the compatible pair $(S_1,S_2)=(\{u_0,u_2\},\{v_5,v_6\})$. The vertical edges
\begin{equation}
    \{A_2B_2,A_3B_3,A_6B_6,A_7B_7,A_8B_8,A_{13}B_{13}\} \subset P
\end{equation}
is determined uniquely by the horizontal edges of $P$.
\end{example}

 Next, we show that $\phi$ is injective by proving the following lemma.

\begin{lemma}
The inverse map $\phi^{-1}$ is well-defined.
\end{lemma}

\begin{proof}
We need to prove that if $P$ is a subgraph of $G_{T,\gamma_{n+3}}$ corresponding to a compatible pair $(S_1,S_2)$ by \eqref{pertocom}, then $P$ is a perfect matching. By the observation above, this is equivalent to showing that every vertex belongs to at most one horizontal edge.\\
\begin{itemize}
    \item If $A_{2i-1}A_{2i}, B_{2i-1}B_{2i} \in P$, then $v_i \in S_2$, so,
    \[u_{i-1},u_{i} \notin S_1 \implies A_{2i-2}A_{2i-1},B_{2i-2}B_{2i-1},A_{2i}A_{2i+1},B_{2i}B_{2i+1} \notin P.\]
    \item If $A_{2i}A_{2i+1}, B_{2i}B_{2i+1} \in P$, then $u_i \in S_1$, so,
    \[u_{v_i},v_{i+1} \notin S_2 \implies A_{2i-1}A_{2i},B_{2i-1}B_{2i},A_{2i+1}A_{2i+2},B_{2i+1}B_{2i+2} \notin P.\]
\end{itemize}
Therefore, the vertices $A_{2i},B_{2i},A_{2i+1},B_{2i+1}$ belong to at most one horizontal edge.
\end{proof}

\begin{proof}[Proof of Theorem]
By the lemmas, the map $\phi$ is injective. Since the cardinalities
\[
\begin{aligned}
    F_{2n+3} &= \text{ number of perfect matching on } G_{T, \gamma_{n+3}} \\
    &= \text{ number of compatible pair on } \mathcal{D}^{(n+1) \times n}
\end{aligned}
\]
are the same by Theorem \ref{clusFib} and Theorem \ref{compFib}, the map $\phi$ gives a one-to-one correspondence with inverse $\phi^{-1}$.
\end{proof}

\subsection{Compatible pairs and nondecreasing Dyck paths}
\label{sec6.2}

Using the notation from Section \ref{sec4} and Section \ref{sec5}, we construct the map
\begin{equation}
    \theta: \{\text{compatible pair on } \mathcal{D}^{(n+1) \times n}\} \to \{\text{nondecreasing Dyck path of length } 2n+4\}
\end{equation}
in four steps as follows.
\begin{itemize}
\item [Step 1:] 
A compatible pair on $\mathcal{D}^{(n+1) \times n}$ is equivalent to a sequence of $n+1$ letters from $\{O, U, V\}$ by the following rules:
\begin{itemize}
    \item if $u_0 \in S_1$, then the first letter is $U$, otherwise, the first letter is $O$,
    \end{itemize}
while for $1\leq i\leq n$, 
\begin{itemize}
    \item if $u_i \notin S_1, v_i \notin S_2$, then the $i^{th}$ letter is $O$,
    \item if $u_i \in S_1, v_i \notin S_2$, then the $i^{th}$ letter is $U$,
    \item if $u_i \notin S_1, v_i \in S_2$, then the $i^{th}$ letter is $V$.
\end{itemize}
Note that the sequence is well-defined because we cannot have $u_i \in S_1$ and $v_i \in S_2$ simultaneously on a compatible pair. Moreover, if $u_i \in S_1$ then $v_{i+1} \notin S_2$, which means that the letter $U$ cannot be followed by the letter $V$.
\item [Step 2:] 
We can always divide a sequence of letters into blocks of the form 
\[[U,...,U,O,V,...,V]\]
(except for the last block, it can be empty or contain only a sequence of $U$'s) such that
\begin{itemize}
    \item each block contains exactly one letter $O$,
    \item there is only a consecutive number of $U$'s (maybe zero) before the letter $O$,
    \item there is only a consecutive number of $V$'s (maybe zero) after the letter $O$.
\end{itemize}
Suppose the number of letter $O$ is $k$, then there are $k+1$ blocks. For $1 \leq i \leq k$, let
\begin{itemize}
    \item $a_i\geq0$ be the number of letters $U$ in the $i^{th}$ block,
    \item $b_i\geq0$ be the number of letters $V$ in the $i^{th}$ block.
\end{itemize}
\item [Step 3:]
Then we construct $k+1$ mountains, where for $1\leq i\leq k$, the $i^{th}$ mountain has magnitude $$(d_i,e_i):=(a_i+b_i+1,a_i+1),$$ while the last mountain has magnitude
\begin{equation}
    (d_{k+1},e_{k+1}):=\left(n+2 - \sum_{i=1}^{k} (a_i+b_i+1),n+2 - \sum_{i=1}^{k} (a_i+1)\right).
\end{equation}
\item [Step 4:] Connect these $k+1$ mountains together, we get a nondecreasing Dyck path of length $2n+4$.
\end{itemize}

\begin{theorem}
The map $\theta$ gives a one-to-one correspondence between the compatible pairs on $\mathcal{D}^{(n+1)\times n}$ and nondecreasing Dyck paths of length $2n+4$.
\end{theorem}

\begin{example}
When $n=6$, $(S_1,S_2)=(\{u_0, u_2\},\{v_5,v_6\})$ is a compatible pair on $\mathcal{D}^{7 \times 6}$, see Figure \ref{fig: compatible}. This compatible pair is equivalent to the sequence of 7 letters
\[(U,O,U,O,O,V,V) .\]
We divide it into four blocks
\[[U,O],\quad [U,O]\quad [O,V,V],\quad [].\]
These blocks correspond to the mountains with magnitudes $(2,2),(2,2),(3,1)$, while the last mountain has magnitude $(1,3)$, see Figure \ref{fig: dyck}.
\end{example}

To prove the theorem, we have the following lemmas.

\begin{lemma}
The map $\theta$ is well-defined.
\end{lemma}

\begin{proof}
The corresponding Dyck path is nondecreasing because for any $1\leq i\leq k$,
\[d_i=a_i+b_i+1 \geq a_i+1=e_i.\]
Furthermore, the corresponding Dyck path has a length of $2n+4$ because
\[\sum_{i=1}^{k+1}d_i = \sum_{i=1}^{k+1} e_i=n+2.\]
\end{proof}

The inverse map $\theta^{-1}$ is determined by reversing the steps of the map $\theta$. From a nondecreasing Dyck path with $k+1$ mountains $(d_i,e_i)$ of length $2n+4$, we have
\[a_i=e_i-1,\quad\quad b_i=d_i-e_i\]
for any $1\leq i\leq k$. Next, we construct $k$ blocks ($1 \leq i \leq k$) 
$$[\;\;\underbrace{U,...,U}_{a_i}\;,O,\;\underbrace{V,...,V}_{b_i}\;\;]$$
while the last block just includes $d_{k+1}-1$ letters $U$. After that, we get a sequence of $d_1+d_2+...+d_k+(d_{k+1}-1)=n+1$ letters. Then there is a compatible pair $(S_1, S_2)$ that is equivalent to this sequence by the correspondence in Step 1 above. 

\begin{lemma}
The map $\theta^{-1}$ is well-defined.
\end{lemma}

\begin{proof}
Note that the first letter in the sequence is $O$ or $U$. Moreover, by the way of connecting the blocks, there is no letter $U$ followed by a letter $V$. So the sequence is indeed induced from a compatible pair.
\end{proof}

\begin{proof}[Proof of Theorem] By the lemmas, the map $\theta$ is injective. Since the cardinalities
\[
\begin{aligned}
    F_{2k+3} &= \text{number of compatible pair on } \mathcal{D}^{(n+1) \times n}\\
    &= \text{number of nondecreasing Dyck path of length } 2n+4
\end{aligned}
\]
are the same by Theorem \ref{compFib} and Theorem \ref{nonDyckFib}, the map $\theta$ gives a one-to-one correspondence with inverse $\theta^{-1}$.
\end{proof}

\subsection{Perfect matchings and nondecreasing Dyck paths}
\label{sec6.3}
The map 
\begin{equation}
    \psi: \{\text{perfect matching on } G_{T, \gamma_{n+3}}\} \to \{\text{nondecreasing Dyck path of length } 2n+4\}
\end{equation}
can be expressed as a composition $\psi = \theta \circ \phi$. However, to make it explicit, we construct it directly instead of using composition. The construction is similar to that of the map $\theta$, by modifying step 1.

\begin{itemize}
\item [Step 1':]
A perfect matching on the snake graph $G_{T, \gamma_{n+3}}$ is equivalent to a sequence of $n+1$ letters from $\{O, U, V\}$ by the following rules:
\begin{itemize}
    \item iIf $A_0A_1,B_0B_1 \in P$, then the first letter is $U$, otherwise, the first letter is $O$,
\end{itemize}
while for $1\leq i\leq n$,
\begin{itemize}
    \item if $A_{2i}B_{2i} \in P$, then the $i^{th}$ letter is $O$,
    \item if $A_{2i}A_{2i+1},B_{2i}B_{2i+1} \in P$, then the $i^{th}$ letter is $U$,
    \item if $A_{2i-1}A_{2i},B_{2i-1}B_{2i} \in P$, then the $i^{th}$ letter is $V$.
\end{itemize}
\end{itemize}
Similar to the arguments given in Section \ref{sec6.2}, $\psi$ is well-defined and gives a one-to-one correspondence between the perfect matchings on the snake graph $G_{T,\gamma_{n+3}}$ associated with annulus with two marked points, and the nondecreasing Dyck paths of length $2n+4$.

\bibliographystyle{abbrv}
\bibliography{bibliography.bib}

\end{document}